\documentclass[oneside]{amsart}

\usepackage{amssymb}
\usepackage{amsmath}
\usepackage{amsthm}

\usepackage{mathrsfs}
\usepackage{mathtools} 
\usepackage{enumerate} 
\usepackage{microtype} 
\usepackage{upgreek}      
\usepackage{geometry}
\usepackage{extarrows}

\usepackage{xcolor}
\usepackage{verbatim}
\usepackage{hyperref}
\usepackage{cleveref}
\usepackage[normalem]{ulem}
\usepackage[utf8]{inputenc}

\usepackage{geometry}               
\usepackage{graphicx}
\usepackage{amssymb}
\usepackage{epstopdf}
\usepackage{verbatim}
\usepackage{mathtools}
\usepackage[english]{babel}
\usepackage{enumerate}
\usepackage{verbatim}
\usepackage{upgreek} 
\newtheorem{theorem}{Theorem}[section]
\newtheorem{prop}[theorem]{Proposition}
\newtheorem{lemma}[theorem]{Lemma}
\newtheorem{cor}[theorem]{Corollary}

\theoremstyle{definition}
\newtheorem{definition}[theorem]{Definition}
\newtheorem{example}[theorem]{Example}

\allowdisplaybreaks
\theoremstyle{remark}
\newtheorem{remark}[theorem]{Remark}
\newtheorem*{remark*}{Remark}

\numberwithin{equation}{section}

\newcommand{\Sph}{\ensuremath{\mathbb{S}}}

\newcommand{\R}{\ensuremath{\mathbb{R}}}
\newcommand{\N}{\ensuremath{\mathbb{N}}}

\newcommand{\G}{\ensuremath{\mathbb{G}}}
\newcommand{\V}{\ensuremath{\mathbb{V}}}

\newcommand{\QV}{\ensuremath{\mathbb{QV}}}

\newcommand{\CW}{\ensuremath{\mathcal{W}}}
\newcommand{\CA}{\ensuremath{\mathcal{A}}}
\newcommand{\CV}{\ensuremath{\mathcal{V}}}
\newcommand{\CH}{\ensuremath{\mathcal{H}}}
\newcommand{\CO}{\ensuremath{\mathcal{O}}}
\newcommand{\CF}{\ensuremath{\mathcal{F}}}
\newcommand{\CE}{\ensuremath{\mathcal{E}}}
\newcommand{\CQ}{\ensuremath{\mathcal{Q}}}

\newcommand{\BE}{\ensuremath{\mathbf{E}}}
\newcommand{\BF}{\ensuremath{\mathbf{F}}}

\newcommand{\bp}{\ensuremath{\mathbf{p}}}
\newcommand{\bq}{\ensuremath{\mathbf{q}}}
\newcommand{\bn}{\ensuremath{\mathbf{n}}}
\newcommand{\bc}{\ensuremath{\mathbf{c}}}

\newcommand{\SD}{\ensuremath{\mathscr{D}}}
\newcommand{\SL}{\ensuremath{\mathscr{L}}}
\newcommand{\SH}{\ensuremath{\mathscr{H}}}

\newcommand{\ud}{\mathrm d}
\newcommand{\uD}{\mathrm D}
\newcommand{\uo}{\mathrm o}

\newcommand{\eps}{\varepsilon}

\DeclareMathOperator{\spt}{spt}
\DeclareMathOperator{\Div}{div}
\DeclareMathOperator{\diam}{diam}
\DeclareMathOperator{\arcosh}{arcosh}
\DeclareMathOperator{\Id}{Id}
\DeclareMathOperator{\tr}{tr}
\DeclareMathOperator{\im}{im}

\makeatletter
\@namedef{subjclassname@2020}{%
	\textup{2020} Mathematics Subject Classification}
\makeatother

\begin{document}
\title{Properties of surfaces with spontaneous curvature}
\author[C.~Scharrer]{Christian Scharrer}
\address{Institute for Applied Mathematics, University of Bonn, Endenicher Allee 60, 53115 Bonn, Germany.}
\email{scharrer@iam.uni-bonn.de}

	
	
\date{\today}

\begin{abstract}
	A model describing cell membranes as optimal shapes with regard to the $L^2$-deficit of their mean curvature to a given constant called spontaneous curvature is considered. It is shown that the corresponding energy functional is lower semi-continuous with respect to oriented varifold convergence on a space of surfaces whose second fundamental form is uniformly bounded in $L^2$. Elementary examples are presented showing that the latter condition is necessary. As a consequence, smoothly embedded minimisers among surfaces of higher genus are obtained. Moreover, it is shown that the diameter of a connected surface is controlled by the $L^1$-deficit of its mean curvature to the spontaneous curvature leading to an improved condition for the existence of minimisers. Finally, the diameter bound can be applied to obtain an isoperimetric inequality.   
\end{abstract}
\maketitle
	
\section{Introduction}\label{sec:intro}
Experiments conducted by Rand--Burton \cite{RandBurton} more than half a century ago show that the membranes of red blood cells deform when sucked into a micropipette. After being released, 
the cells would regain their typical biconcave shape within seconds. The experiments suggest that the shaping forces lie in
the membrane itself and the biconcave form requires the least energy to be maintained. In 1970, Canham \cite{Canham} proposed that the energy minimised is 
\begin{equation*}
	D\Bigl(2\CW(\Sigma) + \int_{\Sigma}K\,\ud\mu\Bigr),\qquad\text{for $\displaystyle \CW(\Sigma)\vcentcolon=\frac{1}{4}\int_{\Sigma}|H|^2\,\ud\mu$,}
\end{equation*}
where $D$ is the bending rigidity, $\Sigma$ is the surface in $\R^3$ describing the membrane, $K$ is the Gauss curvature, $H$ is the trace of the second fundamental form, and $\mu$ is the surface measure. The integral of the Gauss curvature is in fact a topological constant by the Gauss--Bonnet theorem, whereas $\CW$ is well known as the \emph{Willmore functional}. The resulting problem thus can be phrased as follows.

\begin{minipage}{0.9\hsize}
	\vspace{0.2cm}
	\emph{Canham Problem.} Among all shapes of the same given topology, area, and volume, find the one of least Willmore energy.
	\vspace{0.2cm}
\end{minipage}

\noindent In a joint effort 
it was shown that solutions for the Canham problem do exist \cite{Schygulla,KMR,ScharrerNLA,MondinoScharrerACV,KusnerMcGrath}. 

A model for general cell membranes made of lipid bilayer including red blood cells was proposed by Helfrich \cite{Helfrich} in 1973 based on the energy
\begin{equation*}
	2k_c\CH_{c_0}(\Sigma) + \bar k_c\int_\Sigma K\,\ud \mu,\qquad \text{for $\displaystyle\CH_{c_0}(\Sigma)\vcentcolon=\frac{1}{4}\int_{\Sigma}|H-c_0n|^2\,\ud\mu$}
\end{equation*} 
where $c_0\in\R$ is the \emph{spontaneous curvature}, and $k_c,\bar k_c$ are the curvature elastic moduli. The inner unit normal $n$ bringing into play the cell's interior completes the picture. Expanding the square in the \emph{Helfrich functional} $\CH_{c_0}$ gives the expression
\begin{equation*}
	\CH_{c_0}(\Sigma) = \CW(\Sigma) - \frac{c_0}{2}\int_{\Sigma}\langle H,n\rangle\,\ud\mu + \frac{c_0^2}{4}\CA(\Sigma).
\end{equation*}
Just as in the Canham problem, the area $\CA(\Sigma)$ and enclosed volume $\CV(\Sigma)$ are fixed geometric constraints. Hence, the part where the spontaneous curvature enters is the total mean curvature. Recall that this term arises as first variation of the area, thus accounts for small changes of the area under small changes of the surface. Indeed, the parameter $c_0$ is thought to adjust the model with regard to both: small differences in the two layers of the lipid bilayer as well as chemical differences of the two aqueous solutions inside and outside the cell membrane. Each cell membrane thus has its own spontaneous curvature adjusting the model to its specific type and environment. In numerical experiments it was found that the spontaneous curvature of red blood cells is negative and large, see \cite{DeulingHelfrich}. Indeed, if $c_0$ is meant to compensate local concave parts, it must be negative and of the seize of the cell's reciprocal diameter. Since red blood cells are tiny, this number is large. 

While the Helfrich functional takes a neat and simple form, exploring its geometric properties is rather difficult. Many aspects especially in proving existence of solutions with higher genus for the corresponding minimisation problem remain open. This article's contribution is the following. 

\begin{theorem}\label{T}
	Suppose $G$ is a nonnegative integer, $a_0,v_0>0$, and $c_0\in\R$. Denote with $\mathcal S_G$ the set of all compact $2$-dimensional $C^\infty$-submanifolds of $\R^3$ whose genus is at most $G$, and let
	\begin{equation}\label{eq:intro:eta}
		\eta_G(c_0,a_0,v_0)\vcentcolon=\inf\{\CH_{c_0}(\Sigma)\colon \text{$\Sigma\in\mathcal S_G$, $\CA(\Sigma)=a_0$, $\CV(\Sigma)=v_0$}\}.
	\end{equation}
	There exists a number $\Gamma(c_0,a_0,v_0)\in\R$ such that if 
	\begin{equation}\label{eq:intro:eta-Gamma}
		\eta_G(c_0,a_0,v_0) < 8\pi + \Gamma(c_0,a_0,v_0)
	\end{equation}
	then the infimum \eqref{eq:intro:eta} is attained. Moreover,
	\begin{equation}\label{eq:intro:Gamma}
		\Gamma(c_0,a_0,v_0) \ge
		\begin{cases*}
			4\pi\Bigl(\sqrt{1+\frac{|c_0|v_0}{8\pi^2C^2a_0}}-1\Bigr)&\text{for $c_0<0$}\\	
			-6c_0(4\pi^2v_0)^\frac{1}{3}&\text{for $c_0\ge0$}		
		\end{cases*}
	\end{equation}	
	for some universal constant $0<C<\infty$.
\end{theorem}

\begin{remark*}
	In view of \cite[Example 1.2]{RuppScharrer22}, an a priori energy condition such as \eqref{eq:intro:eta-Gamma} is necessary in order to obtain smoothly embedded minimisers. Notice that $\Gamma(c_0,a_0,v_0)>0$ for $c_0<0$. Moreover, for all $c_0\le0$ and $\sigma\ge36\pi$ there exist $\bar a_0,\bar v_0>0$ such that $\bar a_0^3/\bar v_0^2=\sigma$ and $\eta_G(c_0,a_0,v_0)<8\pi$ for all $0<a_0<\bar a_0$, $0<v_0<\bar v_0$ with $a_0^3/v_0^2=\sigma$. Indeed, since 
	\begin{equation*}\label{eq:intro:scaling}
		\lim_{\lambda\to0+}\CH_{c_0}(\lambda\Sigma)=\lim_{\lambda\to0+}\CH_{\lambda c_0}(\Sigma)=\CW(\Sigma),
	\end{equation*}
	this is due to recently obtained example surfaces, see \cite[Theorems 1.1, 1.2]{ScharrerNLA} and \cite{KusnerMcGrath}. In particular, for all $c_0\le0$ and all $G\in\N$, there exist minimisers for some $a_0,v_0>0$. Furthermore, each minimiser solves the Euler--Lagrange equation, thus provides a solution for the fourth order elliptic PDE 
	\begin{equation*}
		\Div\Bigl[\nabla H - \frac{3}{2}\langle H,n\rangle \nabla n + \frac{1}{2}H\times\nabla^\bot n + 2c_0\nabla n + 2\bigl(c_0\langle H,n\rangle - c_0^2 - \alpha\bigr)\nabla \Phi - \rho \Phi \times \nabla^\bot\Phi\Bigr] = 0
	\end{equation*}
	where $\alpha\in\R$ and $\rho\in\R$ are the Lagrange multipliers of area and volume, and $\Phi$ is a local conformal parametrisation of $\Sigma$, cf. \cite[Lemma 4.1]{MondinoScharrer}.
\end{remark*}

Previously, existence of smoothly embedded minimisers was only known for the special case $G=0$ in \cite[Theorem 1.6]{RuppScharrer22} based on \cite{MondinoScharrer}, and for implicitly small $c_0$ in \cite[Corollary 1.4]{ScharrerNLA} and \cite[Theorem 6]{ScharrerPhD} based on \cite{KusnerMcGrath}. The proof of \Cref{T} applies the direct method in the calculus of variations, thus splits into three parts: compactness, lower semi-continuity, and regularity. The first part combines compactness results for Lipschitz immersions \cite[Theorem 1]{ChenLi14AJM} and varifolds \cite{Allard,Mantegazza}, the third part is \cite[Theorem 4.3]{MondinoScharrer} after the breakthrough of \cite{RiviereInvent}, while the second part is provided in \Cref{thm:lsc} of this article. 

\subsection{Lower semi-continuity}
The regularity result \cite[Theorem 4.3]{MondinoScharrer} states that smoothness holds not just for minimisers but for all weak solutions of the Euler--Lagrange equation, i.e. surfaces that are critical under local variations. In order to apply \cite[Theorem 4.3]{MondinoScharrer}, two things must be satisfied. Firstly, the limit and its local variations must be in the same class as the minimising sequence, and secondly, the limit must be regular enough to solve the variational form of the Euler--Lagrange equation. This means the limit must be locally a $W^{2,2}$-immersed disc. The weakest objects the Helfrich energy can be defined for are \emph{oriented varifolds} (Radon measures on $\R^3\times \Sph^2$) with square integrable mean curvature. By a recent result of Bi--Zhou \cite{BiZhou}, unit density varifolds with square integrable mean curvature are locally given as a $W^{2,2}$-embedded disc. Moreover, the \emph{Li--Yau inequality} for the Helfrich functional of \cite{RuppScharrer22} gives a priori energy conditions (phrased in \eqref{eq:intro:eta-Gamma}) that guarantee unit density, and compactness for varifolds is well known by \cite{Allard}. Everything seems in place so far. The only missing bit is lower semi-continuity of $\CH_{c_0}$ with respect to oriented varifold convergence. However,  Röger \cite{Roger} discovered that complete surfaces of constant positive scalar mean curvature found by Grosse-Brauckmann \cite[Remark (ii) on p. 550]{Grosse-Brauckmann} can be used to show that the Helfrich functional is in general not lower semi-continuous with respect to oriented varifold convergence. It is hard to accept that this should be the end of story, especially considering the fact that the counterexample is a sequence of \emph{noncompact} surfaces that only works for positive~$c_0$. Nevertheless, as is done in this article, one can construct sequences of compact smooth surfaces showing that $\CH_{c_0}$ is not lower semi-continuous under oriented varifold convergence, neither for $c_0<0$ nor for $c_0>0$, see Examples \ref{ex:lsc_negative}, \ref{ex:lsc_positive}. These elementary examples consist of two parallel discs with multiple small opposite discs removed where the resulting wholes are filled rotationally symmetric surfaces of constant scalar mean curvature called \emph{nodoids}. Each nodoid forms a bridge connecting the two discs. There are two ways to extend the outer rings of the two discs to a closed surface: closing each disc to a separate sphere, or connecting both discs with a cylinder to one sphere. Each way results in a different sign of the nodoids' scalar mean curvature, giving examples for both $c_0<0$ and $c_0>0$. Picking the nodoids with sectional mean curvature equal to $\pm c_0$, yields that the nodoids do not contribute any energy but take away a significant part of the two discs. The small discs taken away throgh the nodoids return out of nowhere in the limit: the number of needed nodoids, each increasing the genus by~$1$, tends to infinity, while the two discs converge to one disc of multiplicity~$2$. Indeed, as is shown in \Cref{lem:compactness}, lower semi-continuity can only fail if the limit contains a nontrivial part of higher multiplicity. Moreover, the genus and thus the $L^2$-norm of the second fundamental form tends to infinity. It was already suggested by R{\"o}ger \cite{Roger} that imposing a bound on the $L^2$-norm of the second fundamental forms and employing the theory of \emph{curvature varifolds} developed by Hutchinson \cite{Hutchinson86Indiana} and Mantegazza \cite{Mantegazza} might lead to the desired result. Moreover, R{\"o}ger suggested to consider varifolds that are given as the boundary of finite perimeter sets. This seems perfectly natural considering the fact that the varifolds are meant to model cells whose orientation is given by its clearly determined interior. In this article, R{\"o}ger's ideas are combined with the theory of weakly differentiable functions on varifolds developed by Menne \cite{MenneIndiana}. 

Varifolds that arise as limits of boundaries of finite perimeter sets will be called \emph{volume varifolds}. Such varifolds $V$ still have an underlying set of finite perimeter $E\subset\R^3$ assining the volume $\CV(V)=\SL^3(E)$ such that
\begin{equation*}
	\int_E\Div X\,\ud \SL^3 = - \int_{\R^3\times\Sph^2}\langle X(x),\xi\rangle\,\ud V(x,\xi)\qquad \text{for $X\in C^1_c(\R^3,\R^3)$.}
\end{equation*}   
Moreover, each volume varifold is required to be \emph{integral}, meaning there exist an orienting unit normal $\nu$ and nonnegative integer valued $\theta_1,\theta_2$ such that
\begin{equation}\label{eq:intro:integral}
	V(\varphi)=\int_{\R^3}\Bigl[\varphi(x,\nu(x))\theta_1(x) + \varphi(x,-\nu(x))\theta_2(x)\Bigr]\,\ud\CH^2x\qquad\text{for $\varphi\in C^0_c(\R^3\times\Sph^2)$}.
\end{equation}
The idea to prove lower semi-continuity is simple: split the Helfrich energy according to \eqref{eq:intro:integral} into two terms
\begin{equation}\label{eq:intro:Helfrich}
	\CH_{c_0}(V)=\int_{\R^3}|H-c_0\nu|^2\theta_1\,\ud\CH^2 + \int_{\R^3}|H + c_0\nu|^2\theta_2\,\ud\CH^2,
\end{equation}
and prove lower semi-continuity of each term by the standard method used to prove semi-continuity of the Willmore functional, see \eqref{eq:lem:compactness:lsc_Willmore}\eqref{eq:lem:compactness:Willmore}. To do so, each member of the sequence $V_k$ converging to $V$ likewise must be split into two terms each of which converging to one of the terms in \eqref{eq:intro:Helfrich}. The difficulty is caused by parts of higher multiplicity: not even in the smooth case exists a globally continuous choice of the nonunique unit normal $\nu$ in \eqref{eq:intro:integral}. To overcome this problem, the special case where each $V_k$ (but not $V$) has unit density (i.e. $\theta_1+\theta_2\in\{0,1\}$) is treated first. Indeed, the unit normal of each unit density volume varifold with $L^2$-integrable mean curvature is uniquely determined by its underlying set of finite perimeter, and is weakly differentiable, see \Cref{prop:lsc}. Semi-continuity of each term in \eqref{eq:intro:Helfrich} is then proved locally near points $x_0\in\R^3$ where the unit normal $\nu$ of the limit $V$ is approximately continuous. By cutting away what is neither close to $\nu(x_0)$ nor to $-\nu(x_0)$, each $V_k$ splits into two parts, one where the orienting normal is close to $\nu(x_0)$ and one where the orienting normal is close to $-\nu(x_0)$. This way all parts that resemble small nodoids are cut out. Using the coarea formula, the resulting boundary can be controlled in terms of the $L^2$-norm of the second fundamental form which in fact coincides with the $L^2$-norm of the unit normal's weak derivative. Thereby one obtains lower semi-continuity for sequences with uniformly $L^2$-bounded second fundamental forms first for unit density varifolds and then for sequences where each member can be approximated by unit density varifolds. The precise details are given in \Cref{def:qev} and \Cref{thm:lsc}. 

A natural way to realise an $L^2$-bound on the second fundamental form is to impose an upper bound on the genus. This can be done through the theory of Lipschitz immersions (see \Cref{sec:Lipschitz_embeddings}) independently developed by Kuwert--Li \cite{KuwertLi12CAG} and Rivi{\`e}re \cite{Riviere14Crelle}.  

Previously, lower semi-continuity of $\CH_{c_0}$ was only known for the genus zero case with respect to a convergence stronger than mere oriented varifold convergence in \cite[Theorem 3.3]{MondinoScharrer}, and for \emph{minimizing} sequences of smooth immersions of a fixed surface with respect to oriented varifold convergence in \cite[Theorem 1.1]{EichmannAGAG}. The proof in \cite{EichmannAGAG} is based on the $C^{1,\alpha}$-regularity of the limit due to its minimality. There are three improvements in this article compared to \cite{EichmannAGAG}. In \Cref{thm:lsc} sequences of varifolds rather than smooth immersions are considered, the sequence does not have to be minimising, and the limit is in the same class as sequence. As already discussed, the last property is crucial in order to obtain full $C^\infty$-regularity of the minimiser \emph{and} in order to obtain that the minimiser actually solves the Euler--Lagrange equation.

\subsection{Diameter bounds}
An important aspect of \Cref{T} is the lower bound \eqref{eq:intro:Gamma} for the constant $\Gamma(c_0,a_0,v_0)$. For $c_0<0$, this lower bound can be written as
\begin{equation}\label{eq:intro:L}
	4\pi\bigl(\sqrt{1 + L(c_0,a_0,v_0)} - 1\bigr)\qquad \text{for $\displaystyle L(c_0,a_0,v_0)\vcentcolon=\frac{|c_0|v_0}{8\pi^2C^2a_0}$}.
\end{equation}
It results from the Li--Yau inequality of \cite{RuppScharrer22} in order to guarantee embeddedness of the minimising volume varifold $(V,E)$ by estimating its \emph{concentrated volume}
\begin{equation}\label{eq:intro:cvol}
	\CV_c(V,x_0)\vcentcolon=\int_{E}\frac{1}{|x-x_0|^2}\,\ud\SL^3x
\end{equation}
uniformly from below for $x_0$ in the support of the weight measure $\mu_V$. A simple way to do so is to estimate the diameter of $E$ from above. 
A well known  bound for the \emph{intrinsic} (i.e. measured with respect to the geodesic distance) diameter of surfaces $\Sigma\subset\R^3$ was found by Topping \cite{Topping08}:
\begin{equation*}
	d_\mathrm{int}(\Sigma) \le C\int_\Sigma|H|\,\ud\mu.
\end{equation*} 
It is thus an interesting question whether the inequality remains valid if the mean curvature $H$ is replaced by its difference to the spontaneous curvature $H-c_0n$ for $c_0<0$. (For $c_0>0$ this is obviously wrong.) However, as is shown in \Cref{ex:diam}, this is not the case. The example can be understood as sphere packing inside a sphere. All spheres, the packed spheres as well as the container sphere can be obtained as the boundary of a single connected open set defining an inner unit normal $n$. One discovers that the scalar mean curvature $\langle H,n\rangle$ with respect to this inner unit normal is negative for all packed spheres and positive for the container sphere. Picking the diameter of the packed spheres reciprocal to $|c_0|/2$, yields that the packed spheres do not contribute energy. Connecting all spheres by tiny catenoidal bridges results in a connected surface of large intrinsic diameter with moderate energy which makes the desired diameter bound impossible. Luckily, the diameter can be bounded if measured in the ambient space.       

\begin{theorem}\label{T:diam}
	Suppose $c_0\le0$ and $\Sigma$ is a compact connected $2$-dimensional $C^2$-submanifold of $\R^3$. Then its diameter is bounded by
	\begin{equation}\label{eq:intro:diam}
		d(\Sigma)\le C\int_\Sigma|H-c_0n|\,\ud\mu 
	\end{equation} 
	for some universal constant $0<C<\infty$.
\end{theorem}
The statement is in fact valid for all volume varifolds (see \Cref{thm:diam}). Using this diameter bound to estimate the concentrated volume \eqref{eq:intro:cvol} results in the constant $L(c_0,a_0,v_0)$ in \eqref{eq:intro:L}, cf. \Cref{rem:cor:lsc}. Previously, based on Simon's diameter bound \cite[Lemma 1.1]{SimonCAG}, it was only known that
\begin{equation*}
	L(c_0,a_0,v_0) \ge \frac{|c_0|v_0}{2\cdot9^2(|a_0| + \frac{2}{3}|c_0|v_0)},
\end{equation*}
see \cite[Theorem 1.6, Lemma 6.4]{RuppScharrer22}. As an improvement, it is now known that $L(c_0,a_0,v_0)$ grows at least linearly in $c_0$ and $v_0$.

Considering the fact that Topping's proof \cite{Topping08} gives an upper bound for the \emph{intrinsic} diameter, a different approach is needed to prove \Cref{T:diam}. The idea is to write the right hand side of the Inequality \eqref{eq:intro:diam} as the first variation of a (nonrectifiable and unoriented) $2$-varifold (i.e.\ a Radon measure on the Cartesian product of $\R^3$ with the Grassmannian manifold $\G(3,2)$ of $2$-dimensional linear subspaces of $\R^3$). To be more precise, let $E$ be the open bounded subset of $\R^3$ with $\Sigma = \partial E$, let $V$ be the $2$-varifold induced by $\Sigma$, and let $W$ be the $2$-varifold in $\R^3$ defined by 
\begin{equation*}
	W\vcentcolon=(\SL^3\llcorner E)\times\sum_{i=1}^3\frac{\delta_{P_i}}{2}\qquad\text{for $P_i=\{(x_1,x_2,x_3)\in\R^3\colon x_i=0\}$}.
\end{equation*}
Then, since $c_0\le0$,
\begin{equation*}
	V_{c_0}\vcentcolon= V - c_0W
\end{equation*}
defines a $2$-varifold in $\R^3$. Since $W$ originates from a $3$-dimensional object artificially downgraded to a $2$-dimensional varifold, it contributes an unrectifiable part to $V_{c_0}$. Such functionals adding area and volume naturally appear in the study of constant mean curvature surfaces as their first variation is given by
\begin{equation*}
	\updelta W_{c_0}(X)\vcentcolon=\int_{\R^3\times\G(3,2)} \Div_P X(x)\,\ud V_{c_0}(x,P)=-\int_{\Sigma}\langle X, H-c_0 n\rangle\,\ud\mu \qquad \text{for $X\in C^1_c(\R^3,\R^3)$},
\end{equation*}
where $\Div_P X\vcentcolon= \langle \uD X(b_1),b_1\rangle + \langle \uD X(b_2),b_2\rangle$ for any orthonormal basis $(b_1,b_2)$ of $P$, cf. \Cref{prop:vv}\eqref{it:prop:vv:first_variation}. Using this relation, one can prove lower density bounds for points of the rectifiable part (see \Cref{lem:lower_density-bound}). Following a method of Menne \cite{MenneDiam}, these can then be used to prove \Cref{T:diam}. 

\subsection{Isoperimetric inequalities}
Adapting a method of Topping \cite{Topping08} inspired by Simon's monotonicity identity in \cite{SimonCAG}, the following isoperimetric inequality for surfaces $\Sigma\subset\R^3$ with enclosed open set $E$ and spontaneous curvature $c_0\le0$ is proved in \Cref{sec:iso}:
\begin{align*}
	\pi\CA(\Sigma) &\le \Bigl(\int_{\Sigma}|H-c_0n|\,\ud \mu\Bigr)^2 +c_0\int_{E}\int_{\Sigma}\frac{1}{|x-y|^2}\,\ud\mu x\,\ud\SL^3y \\
	&\qquad + 2c_0 \int_{E}\int_{\Sigma}\frac{|H(x)-c_0 n(x)|}{|x-y|}\,\ud\mu x\,\ud\SL^3y.
\end{align*}
Interestingly enough, the integrals resemble an energy in liquid drop models, see for instance \cite[Equation (1)]{ChoksiMuratovTopaloglu}. Estimating the last term by means of the diameter bound \eqref{eq:intro:diam} leads to 
\begin{equation*}
	\pi \CA(\Sigma) + 2|c_0|C^{-1}\SL^3(E) +|c_0|\int_{E}\int_{\Sigma}\frac{1}{|x-y|^2}\,\ud\mu x\,\ud\SL^3y\le \Bigl(\int_{\Sigma}|H-c_0n|\,\ud \mu\Bigr)^2. 
\end{equation*}

\subsection*{Acknowledgements}
The author would like to thank Tim Laux, Ulrich Menne, and Stefan M{\"u}ller for their interest and helpful advice.

\section{Preliminaries}\label{sec:prelim}
This section is devoted to recall the notion of \emph{oriented varifold} and related definitions as well as to fix some notation. The main references are \cite{Federer,Allard,Hutchinson86Indiana}.

\subsection{Functions and Measures}\label{sec:functions}
Here, some general notations and definitions will be presented.

Let $(M,d)$ be a metric space. The \emph{diameter} of a set $A\subset M$ is defined as
\begin{equation*}
	\diam A \vcentcolon=\sup\{d(x,y)\colon x,y\in A\}.
\end{equation*}
The topological boundary, interior, and closure of $A$ in $M$ are denoted with $\partial A$, $A^\circ$, and $\overline A$. The \emph{image} and the \emph{support} of a function $f\colon M\to \R^n$ are given by
\begin{equation*}
	\im f\vcentcolon=f(M),\qquad \spt f\vcentcolon=\overline{\{x\in M\colon f(x)\ne0\}}.
\end{equation*}
Define the open and closed balls corresponding to $a\in M$ and $r>0$ by
\begin{equation*}
	B_r(a)\vcentcolon=\{x\in M\colon d(a,x) < r \}, \qquad \bar B_r(a)\vcentcolon=\{x\in M\colon d(a,x) \le r \}
\end{equation*}
and abbreviate $B_r\vcentcolon =B_r(0), \bar B_r\vcentcolon=\bar B_r(0)$ in case $M$ is a subset of a normed vector space. A sequence of Radon measures $\mu_k$ over $M$ is said to converge to $\mu$ if and only if $\mu$ is a Radon measure over $M$ and
\begin{equation}\label{eq:weak_convergence}
	\lim_{k\to\infty} \int_Mf\,\ud\mu_k = \int_Mf\,\ud \mu =\vcentcolon\mu(f)
\end{equation}
for all continuous functions $f\colon M\to\R$ with compact support. Let $\mu$ be a measure over $M$ (as defined in \cite[2.1.2]{Federer}) and $m$ be a positive integer. The $m$-dimensional \emph{density} of $\mu$ at $a\in M$ is defined by
\begin{equation}\label{eq:density}
	\Theta^m(\mu,a)\vcentcolon= \lim_{r\to0+}\frac{\mu(\bar B_r(a))}{\boldsymbol{\alpha}(m)r^m},
\end{equation}
provided the limit exists and is finite, where $\boldsymbol{\alpha}(m)\vcentcolon=\SL^m(\R^m\cap\bar B_1)$ and $\SL^m$ is the $m$-dimensional Lebesgue measure. The \emph{support} of $\mu$ is defined as
\begin{equation*}
	\spt\mu \vcentcolon = M \setminus\{x\in M\colon\mu(B_r(x))=0\text{ for some }r>0\}.
\end{equation*}
Given any set $N$ and a map $f\colon M\to N$, the \emph{push forward} of $\mu$ under the map $f$ is defined by
\begin{equation*}
	(f_\#\mu)(B)\vcentcolon=\mu(f^{-1}(B))\qquad \text{for all $B\subset N$}.
\end{equation*}
Given any set $A\subset M$ and $0\le f\in L^1_\mathrm{loc}(\mu)$, the measures $\mu\llcorner A$ and $\mu\llcorner f$ over $M$ are defined by
\begin{equation*}
	(\mu\llcorner A)(S) \vcentcolon= \mu(A\cap S),\qquad (\mu\llcorner f)(S) = \int_S^*f\,\ud\mu,
\end{equation*}
where $\int^*_S$ is the upper integral in case $S$ is not $\mu$-measurable, cf.\ \cite[2.4.10]{Federer}. The \emph{characteristic function} $\chi_A\colon M\to\R$ and the \emph{identity map} $\Id_M\colon M\to M$ are defined by
\begin{equation*}
	\chi_A(x)\vcentcolon=
	\begin{cases*}
		1 & $x\in A$ \\
		0 & $x\in M\setminus A$,
	\end{cases*}
	\qquad \Id_M(x)\vcentcolon=x.
\end{equation*}
The \emph{Dirac measure} on $M$ of any point $x\in M$ is denoted with $\delta_x$. Moreover, the $m$-dimensional Hausdorff measure on $\R^3$ with respect to the Euclidean metric is denoted with $\mathscr H^m$.

Given $U\subset \R^m$ open, the set of $\R^n$-valued $k$-times continuously differentiable functions on $U$ is denoted with $C^k(U,\R^m)$. The set of all $\R^n$-valued continuously differentiable functions on $\R^m$ with compact support in $U$ is denoted with $C^k_c(U,\R^n)$.

The Euclidean norm and inner product on $\R^n$ are denoted with $|\cdot|$ and $\langle\cdot,\cdot\rangle$.
Let $A\subset \R^n$ and $b\in\R^n$. Then $u\in\Sph^{n-1}\vcentcolon=\{x\in\R^n\colon|x|=1\}$ is called an \emph{inner normal} of $A$ at $b$, if and only if
\begin{equation*}
	\Theta^n(\SL^n\llcorner\{x\in A\colon\langle x-b,u\rangle<0\},b)=0=\Theta^n(\SL^n\llcorner\{x\in \R^n\setminus A\colon\langle x-b,u\rangle>0\},b).
\end{equation*}
In this case, $-u$  is an exterior normal of $A$ at $b$ according to the definition \cite[4.5.5]{Federer}. There exists at most one inner normal of $A$ at $b$, leading to the definition
\begin{equation*}\label{eq:inner_normal}
	\bn_A(b)\vcentcolon=
	\begin{cases}
		u & \text{$u$ is an inner normal of $A$ at $b$},\\
		0 & \text{there exists no inner normal of $A$ at $b$}. 
	\end{cases}
\end{equation*}

\subsection{Currents and oriented varifolds}\label{sec:currents}
Following Federer \cite{Federer}, Allard \cite{Allard}, and Hutchinson \cite{Hutchinson86Indiana}, basic notations and definitions for currents and oriented varifolds will be introduced.

Let $m\leq 3$ be a nonnegative integer. The vector space of \emph{$m$-vectors} is given by the linear span of the \emph{simple} $m$-vectors 
\begin{equation*}
	\textstyle\bigwedge_m \R^3 \vcentcolon= \mathrm{span} \{v_1\wedge\ldots\wedge v_m\colon v_1,\ldots, v_m\in\R^3 \}
\end{equation*}
for $m\geq1$ and $\bigwedge_0 \R^3\vcentcolon=\R$, where the wedge product is uniquely determined by the properties
\begin{equation*}
	(u+\alpha v)\wedge w = (u\wedge w) + \alpha(v\wedge w),\quad (u\wedge v)\wedge w = u\wedge(v\wedge w),\quad u\wedge u=0
\end{equation*} 
for $u,v,w\in\R^3$, $\alpha\in\R$.
Obviously, $\bigwedge_1 \R^3=\R^3$ and $\bigwedge_3 \R^3\simeq\R$. 
The Gram determinant $\det \bigl(\langle v_i,w_j\rangle_{i,j=1}^{m}\bigr)=\vcentcolon\langle v,w\rangle$ for $v=v_1\wedge\ldots\wedge v_m$ and $w=w_1\wedge\ldots\wedge w_m$ turns $\bigwedge_m \R^3$ into an inner product space 
with respect to which the hodge star operator
\begin{equation*}
	\star\colon\textstyle\bigwedge_m \R^3 \to \textstyle\bigwedge_{3-m} \R^3
\end{equation*}
is a linear isometry satisfying ${\star{\star v}} = v$ for $v\in\R^3$ and $\star(v_1\wedge v_2) = v_1\times v_2$, where $\times$ denotes the usual cross product on $\R^3$. In particular, $\bigwedge_2 \R^3\simeq\R^3$.

The vector space of \emph{alternating $m$-forms} $\textstyle\bigwedge^m \R^3$ consists of all real valued linear maps on $\textstyle\bigwedge_m\R^3$. 
For $U\subset\R^3$ open, the vector space $\mathscr D_m(U)$ of $m$-dimensional \emph{currents} in $U$ consists of all linear functionals $T\colon C^\infty_c(U,\textstyle\bigwedge^m \R^3) \to \R$ such that for all $N\in\N$ and all compact $K\subset U$ there exists $C<\infty$ with
\begin{equation}\label{eq:distribution}
	|T(\phi)|\le C \sum_{k=0}^N\sup_{x\in K}\|\uD^k\phi_x\|.
\end{equation}
Denoting with $\ud\colon\SD_m(\R^3)\to\SD_{m+1}(\R^3)$ for $m\le 2$ the exterior derivative, the \emph{boundary operator} for $m\ge1$ is defined by
\begin{equation*}
	\partial \colon \SD_m(\R^3)\to\SD_{m-1}(\R^3), \qquad \partial T(\omega) \vcentcolon= T(\ud\omega).
\end{equation*} 
For all $\SL^3$-measurable sets $A$, the current $\BE^3 \llcorner A\in\SD_3(\R^3)$ is defined by
\begin{equation*}\label{eq:Euclidean_current}
	(\BE^3\llcorner A)(\omega)\vcentcolon=\int_A\omega_x(e_1\wedge e_2\wedge e_3)\,\ud\SL^3x,
\end{equation*}
where $\{e_1,e_2,e_3\}$ is the standard basis of $\R^3$.

The metric space of oriented $2$-dimensional subspaces of $\R^3$ is denoted with
\begin{equation*}
	\G^\uo(3,2)\vcentcolon=\{\xi\in \textstyle\bigwedge_2 \R^3\colon \xi \text{ is simple, } |\xi|=1\}. 
\end{equation*}
Of course, the image of $\G^\uo(3,2)$ under the hodge star operator is just the unit $2$-sphere $\mathbb S^2\subset \R^3$. In particular, $\G^\uo(3,2)$ is a closed $2$-dimensional smooth manifold, thus a locally compact Hausdorff space. The image of $\G^\uo(3,2)$ under the map
\begin{equation}\label{eq:projection}
	\mathbf p(\xi)\vcentcolon=\{v\in\R^3\colon v\wedge\xi=0\}
\end{equation}
is denoted with $\G(3,2)$ and endowed with the metric
\begin{equation}\label{eq:metric}
	d(S,P)\vcentcolon=\min\{|\zeta-\xi|\colon \bp(\zeta)=S,\,\bp(\xi)=P \}
\end{equation}
with respect to which the projection $\bp$ has Lipschitz-constant $1$.
Given $P\in\G(3,2)$ and $\xi\in\G^\uo(3,2)$ with $\bp(\xi)=P$, the linear map $P_\natural\colon\R^3\to\R^3$ is defined by $P_\natural(v)=v-\langle v,\star\xi\rangle(\star\xi)$. For $U\subset \R^3$ open, the set of Radon measures over $\G^\uo_2(U)\vcentcolon=U\times\G^\uo(3,2)$, respectively $\G_2(U)\vcentcolon=U\times\G(3,2)$, is denoted with $\V_2^\uo(U)$, respectively $\V_2(U)$. Its members are referred to as \emph{oriented $2$-varifolds}, respectively \emph{$2$-varifolds}, in $U$. Given any $V\in\V_2(U)$, its \emph{weight measure} $\mu_V$ is defined as the push forward of $V$ under the projection $U\times\G(3,2)\to U$. For all vector fields $X\in C^1(U,\R^3)$ and all $\xi\in\G^\uo(3,2)$ let 
\begin{equation*}
	\Div_{\mathbf p(\xi)}X(x) \vcentcolon= \Div X(x) - \langle\uD X_x(\star\xi),\star\xi\rangle. 
\end{equation*}
The \emph{first variation} of any $2$-varifold $V$ in $U$ is defined by
\begin{equation}\label{eq:first_variation}
	\updelta V\colon C^\infty_c(U,\R^3)\to\R, \qquad \updelta V(X)\vcentcolon=\int_{\G_2(U)}\Div_PX(x)\,\ud V(x,P).
\end{equation}
A linear map $T\colon C^\infty_c(U,\R^3)\to\R$ satisfying \eqref{eq:distribution} is said to be \emph{representable by integration}, if the induced Borel regular measure uniquely determined by
\begin{equation*}\label{eq:total_variation_measure}
	\|T\|(A)\vcentcolon=\sup\{T(X)\colon X\in C^\infty_c(A,\R^3),\,|X|\le1\}\qquad\text{for all open $A\subset U$}
\end{equation*}
is a Radon measure. In that case there exists $\eta\in L^1_\mathrm{loc}(\|T\|,\Sph^{2})$ such that 
\begin{equation}\label{eq:total_variation}
	T(X) = \int_U\langle X,\eta\rangle\,\ud\|T\|=\vcentcolon(\|T\|\llcorner \eta)(X),
\end{equation}
see \cite[4.1.5]{Federer}. 

A sequence $V_k$ of $2$-varifolds, respectively oriented $2$-varifolds in $U$ is said to converge in $\V_2(U)$, respectively in $\V_2^\uo(U)$, if $V_k$ converge as Radon measures according to \eqref{eq:weak_convergence}.

For each $V\in\V_2^\uo(U)$ there naturally corresponds a current $\mathbf c(V)\in\SD_2(U)$ defined by
\begin{equation}\label{eq:induced_current}
	\mathbf c(V)(\omega) \vcentcolon=\int_{\G_2^\uo(\R^3)}\omega_x(\xi)\,\ud V(x,\xi).
\end{equation}
Moreover, there corresponds an unoriented varifold given by the push forward of $V$ under the map $\mathbf q(x,\xi)\vcentcolon=(x,\mathbf p(\xi))$. Let $\mu_V\vcentcolon=\mu_{\mathbf q_\#V}$ and $\updelta V\vcentcolon=\updelta(\mathbf q_\#V)$.

\section{Volume varifolds}\label{sec:vv}
In this section, the notion of volume varifold (\Cref{def:vv}) alongside with related functionals (\Cref{def:functionals}) are introduced. Subsequently, a list of their basic properties (\Cref{prop:vv}) including compactness (\Cref{lem:compactness}) will be presented.
\begin{definition}\label{def:vv}
	A \emph{volume varifold} is a pair $(V,E)$ consisting of an oriented $2$-varifold $V$ in $\R^3$ and an $\SL^3$-measurable set $E\subset\R^3$ such that the following hypotheses hold. \medskip
	\begin{enumerate}[(a)]
		\item \label{it:def:vv:mass} $\mu_V(\R^3) < \infty$;\medskip
		
		\item \label{it:def:vv:mc} $\updelta V(X) = -\int_{\R^3}\langle X,H\rangle\,\ud\mu_V$ \quad for some $H\in L^2(\mu_V,\R^3)$;\medskip
		
		\item \label{it:def:vv:density} $\Theta^2(\mu_V,x)\ge1$ \quad for $\mu_V$-almost all $x\in\R^3$; \medskip
		
		\item \label{it:def:vv:integral} there exist a Borel function $\nu\colon\spt\mu_V\to\G^\uo(3,2)$ and integer valued $\theta_1,\theta_2\in L^1(\mathscr H^2\llcorner \spt\mu_V)$ such that 
		\begin{equation*}\label{eq:def:vv:integral}
			V(\varphi) = \int_{\spt\mu_V}\Bigl[\varphi(x,\nu(x))\theta_1(x) + \varphi(x,-\nu(x))\theta_2(x)\Bigr]\,\ud\mathscr H^2x
		\end{equation*}
		for all continuous functions $\varphi\colon\G_2^\uo(\R^3)\to\R$; \medskip
				
		\item \label{it:def:vv:divergence} $\int_E\Div X\,\ud\SL^3 = -\int_{\G_2^\uo(\R^3)}\langle X(x),\star\xi\rangle\,\ud V(x,\xi)$ \quad for all $X\in C^1_c(\R^3,\R^3)$;\medskip
		
		\item \label{it:def:vv:boundedness} $\SL^3(E)<\infty$. \medskip 
	\end{enumerate}
\end{definition}

\begin{remark}\label{rem:def:vv}\leavevmode 
	\begin{enumerate}[(1)]
		\item A similar notion had been introduced in \cite[Hypothesis 4.5]{RuppScharrer22} as \emph{varifolds with enclosed volume}, cf.\ \Cref{prop:vv}\eqref{it:prop:vv:perpendicular}\eqref{it:prop:vv:diam}.
		
		\item The set $E$ has finite perimeter in the sense of \cite[4.5.1]{Federer}. Its perimeter $\|\partial (\BE^3\llcorner E)\|(\R^3)$ is bounded above by $\mu_V(\R^3)$, see \Cref{prop:vv}\eqref{it:prop:vv:perimeter}. 
		
		\item In view of \Cref{prop:vv}\eqref{it:prop:vv:restriction}, one may equivalently require that the set $E$ is open or likewise closed.

		\item The sign convention in \eqref{it:def:vv:divergence} means that the vector $\star\xi$ plays the role of an inner normal.
		
		\item \label{it:rem:def:vv:rectifiability} If $V\in\V_2^\uo(\R^3)$ satisfies Hypotheses \eqref{it:def:vv:mass}\eqref{it:def:vv:mc}\eqref{it:def:vv:density}, then $\bq_\# V$ is rectifiable by \cite[Theorem 5.5(1)]{Allard}. In particular, denoting with $P(x)$ the $\mu_V$-approximate tangent plane at $x\in\R^3$ (cf.\ \Cref{rem:def:lsc}), it follows that 
		\begin{equation*}\label{eq:rem:def:vv}
			V(\{(x,\xi)\in\G_2^\uo(\R^3)\colon \bp(\xi) \neq P(x)\})=0.	
		\end{equation*}
		Therefore, given any Borel map $i\colon \G(3,2)\to\G^\uo(3,2)$ with $\bp\circ i = \Id_{\G(3,2)}$, say $i=(\bp|_{A})^{-1}$ for 
		\begin{equation*}
			A\vcentcolon=\star\Bigl( \{(x_1,x_2,x_3)\in\Sph^2\subset\R^3\colon x_3\ge0\}\setminus\{(\cos(t),\sin(t),0)\colon 0 \le t <\pi\}\Bigr),
		\end{equation*}
		one can apply the theory of symmetrical derivation for measures \cite[2.9.10]{Federer} 
		to deduce the existence of possibly non integer valued $\theta_1, \theta_2$ that satisfy Hypothesis \eqref{it:def:vv:integral} for $\nu = i\circ P$. As \Cref{ex:non-integral} shows, requiring $\Theta^2(\mu_V,\cdot)$ to be integer valued is not enough to deduce that also the functions $\theta_1,\theta_2$ are integer valued.   
		
		\item \label{it:rem:def:vv:integral} Of course, \cite[Theorem 3.5(1)(c)]{Allard} and Hypothesis \eqref{it:def:vv:integral} imply that the varifold $\mathbf q_\#V$ is integral, cf.\ \Cref{prop:vv}\eqref{it:prop:vv:integral}. 
	\end{enumerate}
\end{remark}

\begin{example}\label{ex:non-integral}
	Define the oriented $2$-varifold $V$ in $\R^3$ by
	\begin{equation*}
		V(\varphi) = \frac{1}{2}\int_{\Sph^2}\varphi(x,\star x) + \varphi(x,-(\star x))\,\ud\SH^2x.
	\end{equation*}
	Then, $\mu_V = \SH^2\llcorner \Sph^2$, $\Theta^2(\mu_V,x)=1$ for all $x\in\Sph^2$, and $\updelta V(X) = - \int_{\Sph^2}\langle X,H\rangle\,\ud\SH^2$ for $H(x) = -2x$. In particular, $V$ satisfies Hypotheses \eqref{it:def:vv:mass}\eqref{it:def:vv:mc}. Moreover, $V$ satisfies \eqref{it:def:vv:divergence}\eqref{it:def:vv:boundedness} for $E=\varnothing$. However, any pair $\theta_1,\theta_2$ that satisfies the equation in \eqref{it:def:vv:integral} is given by $\theta_1(x)=\theta_2(x)=1/2$, so neither $\theta_1$ nor $\theta_2$ are integer valued. Thus, $V$ is not a volume varifold. 
\end{example}

\begin{definition}\label{def:functionals}\leavevmode
	\begin{enumerate}[(1)]
		\item \label{it:def:functionals:F} For all $\SL^3$-measurable sets $A$, the varifold $\BF^2\llcorner A\in\V_2(\R^3)$ is defined by
		\begin{equation*}
			\BF^2\llcorner A\vcentcolon=(\SL^3\llcorner A)\times\sum_{i=1}^3\frac{\delta_{E_i}}{2}\qquad\text{for $E_i=\{(x_1,x_2,x_3)\in\R^3\colon x_i=0\}$}.
		\end{equation*}
		\item For all $x\in\R^3$ and $S\in\G(3,2)$,
		\begin{equation*}
			x\perp S \quad \vcentcolon\iff\quad \langle x,v\rangle = 0 \text{ for all $v\in S$}.
		\end{equation*}
		\item For $X\colon\R^3\to\R^3$ let $X^\top,X^\bot\colon \G^\uo_2(\R^3)\to\R^3$ be defined by $X^\bot(x,\xi)\vcentcolon=\langle X(x),\star\xi\rangle(\star\xi)$ and $X^\top\vcentcolon= X - X^\bot$.
		\item The \emph{area functional} on $\V_2^\uo(\R^3)$ and $\V_2(\R^3)$ is defined by $\CA(V)\vcentcolon=\mu_V(\R^3)$.
		\item The \emph{volume} of a volume varifold $(V,E)$ is defined as $\CV(V)\vcentcolon=\SL^3(E)$.
		\item \label{it:def:functionals:Willmore} The \emph{Willmore functional} $\CW$ on $\V_2^\uo(\R^3)$ and $\V_2(\R^3)$ is defined by
		\begin{equation*}
			\sqrt{4\CW(V)}\vcentcolon=\sup\{\updelta V(X)\colon X\in C^\infty_c(\R^3,\R^3),\,\|X\|_{L^2(\mu_V)}\leq1\}.
		\end{equation*}
		\item \label{it:def:functionals:Helfrich} For all $V\in\V_2^\uo(\R^3)$ satisfying $\updelta V(X) = -\int_{\R^3}\langle X,H\rangle\,\ud\mu_V$ for some $H\in L^1_\mathrm{loc}(\mu_V,\R^3)$, the \emph{Helfrich energy} is defined as
		\begin{equation*}
			\CH_{c_0}(V)\vcentcolon=\frac{1}{4}\int_{\G_2^\uo(\R^3)}|H(x)-c_0(\star\xi)|^2\,\ud V(x,\xi)\qquad  \text{for $c_0\in\R$}.
		\end{equation*} 
		The number $c_0$ is referred to as \emph{spontaneous curvature}.
		\item \label{it:def:functionals:concentrated_volume} For all volume varifolds $(V,E)$, and all $x_0\in\R^3$, the \emph{concentrated volume} of $V$ at $x_0$ is defined by
		\begin{equation*}
			\CV_c(V,x_0)\vcentcolon=\int_{E}\frac{1}{|x-x_0|^2}\,\ud\SL^3x.
		\end{equation*}
	\end{enumerate}
\end{definition}

\begin{remark}\label{rem:def:functionals}
	For $V\in\V_2^\uo(\R^3)\cup\V_2(\R^3)$ to have finite Willmore energy, it is necessary and sufficient that there exists a $\mu_V$ almost unique $H\in L^1_\mathrm{loc}(\mu_V,\R^3)\cap L^2(\mu_V,\R^3)$ such that
	\begin{equation*}\label{eq:rem:def:functionals:mc}
		\updelta V(X) = -\int_{\R^3}\langle X,H\rangle\,\ud \mu_V,
	\end{equation*}
	see for instance \cite[3.4]{MS23a}. In that case, $\CW(V) = \frac{1}{4}\|H\|_{L^2(\mu_V)}^2$.
\end{remark}

\begin{prop}\label{prop:vv}
	Suppose $(V,E)$ is a volume varifold, $\theta_1,\theta_2,\nu$ are as in \Cref{def:vv}\eqref{it:def:vv:integral}, $\alpha \geq 0$, $W\vcentcolon=\BF^2\llcorner E$, and $V_\alpha \vcentcolon = \alpha W + \mathbf q_\#V$. Then the following hold. \medskip
	\begin{enumerate}[\upshape(1)]
		\item \label{it:prop:vv:perimeter} $\partial (\BE^3\llcorner E) = -\mathbf c(V), \quad \partial \mathbf c(V) = 0, \quad \|\partial (\BE^3\llcorner E)\|(\R^3) \le \CA(V)$; \medskip
		\item \label{it:prop:vv:iso-ineq} $\CV(V)\leq C\CA(V)^\frac{3}{2}$ \quad for some universal constant $C<\infty$; \medskip
		\item \label{it:prop:W_by_H} $\CW(V) \leq 2\CH_{c_0}(V) + \frac{c_0^2}{2}\CA(V)$ \quad for all $c_0\in\R$; \medskip
		\item \label{it:prop:H_by_W} $\CH_{c_0}(V) \leq 2\CW(V) + \frac{c_0^2}{2}\CA(V)$ \quad for all $c_0\in\R$; \medskip		
		\item \label{it:prop:vv:integral} $\theta_1,\theta_2\ge0, \quad \Theta^2(\mu_V,x)=\theta_1(x)+\theta_2(x)$ for $\SH^2$-almost all $x \in \spt\mu_V$; \medskip 
		\item \label{it:prop:vv:perpendicular} $H(x) \perp \mathbf p(\xi)$ \quad for $V$-almost all $(x,\xi)\in\G_2^\uo(\R^3)$; \medskip
		\item \label{it:prop:vv:support} $\Theta^2(\mu_V,x)\geq1$ for all $x\in\spt\mu_V$, \quad $\mathscr H^2(\spt\mu_{V}) \le \CA(V) < \infty$;\medskip
		\item \label{it:prop:vv:indecomposability} $(\bq_\#V)\llcorner C\times \G(3,2)$ satisfies Hypotheses \eqref{it:def:vv:mass}\eqref{it:def:vv:mc}\eqref{it:def:vv:density}, and is indecomposable of type $C_c^\infty(\R^3,\R)$ (in the sense of \cite[Definition 7.1]{MS23a}) whenever $C$ is a connected component of $\spt\mu_V$; \medskip
		\item \label{it:prop:vv:rectifiability} $\Theta^2(\mu_V,\cdot)\le \frac{1}{4\pi}\CW(V)$, \quad $\spt\mu_V$ is compact and $\mathscr H^2$-rectifiable (in the sense of \cite[3.2.14(4)]{Federer});\medskip
		\item \label{it:prop:vv:density-W} $\mu_{W} = \frac{3}{2}\SL^3\llcorner E, \quad\Theta^2(\mu_{W},x) = 0$ for all $x\in\R^3$; \medskip
		\item \label{it:prop:vv:density-V} $\mu_{V_\alpha} = \mu_V + \alpha\frac{3}{2}\SL^3\llcorner E,\quad \Theta^2(\mu_{V_\alpha},x) = \Theta^2(\mu_V,x)$ for all $x\in\R^3$; \medskip
		\item \label{it:prop:vv:normal} $\bn_E(x)={\star\nu}(x)(\theta_1(x)-\theta_2(x))$ \quad for $\SH^2$-almost all $x\in \spt\mu_V$; \medskip  
		\item \label{it:prop:vv:even} $\theta_1(x) = \theta_2(x)$, and $\Theta^2(\mu_V,x)$ is even \quad for $\SH^2$-almost all $x\in\{\mathbf n_E = 0\}$; \medskip
		\item \label{it:prop:vv:odd} $|\theta_1(x)-\theta_2(x)|=1$, and $\Theta^2(\mu_V,x)$ is odd \quad for $\SH^2$-almost all $x\in \{\mathbf n_E \neq 0\}$; \medskip
		\item \label{it:prop:vv:boundary} $\partial \spt \SL^3\llcorner E \subset \spt \SH^2\llcorner \{\bn_E\neq0\} \subset \spt\mu_V$; \medskip
		\item \label{it:prop:vv:restriction} $\SL^3\llcorner E= \SL^3\llcorner S_E = \SL^3\llcorner (S_E)^\circ$ \quad for $S_E = \spt \SL^3\llcorner E$; \medskip 
		\item \label{it:prop:vv:diam} $\diam \spt \SL^3\llcorner E \le \diam \spt \mu_V$; \medskip
		\item \label{it:prop:vv:F} $\updelta W(X) = \int_{E}\Div X\,\ud\SL^3$; \medskip
		\item \label{it:prop:vv:first_variation} $\updelta V_\alpha(X) = -\int_{\G_2^\uo(\R^3)}\langle X(x), H(x) + \alpha(\star\xi)\rangle\,\ud V(x,\xi)$;\medskip
		\item \label{it:prop:vv:mc} $\updelta V_\alpha(X) = -\int_{\spt\mu_V}\langle X,H_\alpha\rangle\,\ud\mu_{V_\alpha}$ \quad for $H_\alpha\vcentcolon = H +\alpha\frac{\theta_1-\theta_2}{\theta_1+\theta_2}(\star\nu)$; \medskip
		\item \label{it:prop:vv:perpendicular-F} $H_\alpha(x) \perp S$ \quad for $V_\alpha$-almost all $(x,S)\in \G_2(\R^3)$; \medskip
		\item \label{it:prop:vv:Helfrich-Willmore} $\CH_{-\alpha}(V) = \CW(V_\alpha)+\frac{\alpha^2}{4}\int_{\spt\mu_V}\bigl[1-\bigl(\frac{\theta_1-\theta_2}{\theta_1+\theta_2}\bigr)^2\bigr]\,\ud\mu_V$; \medskip
		\item \label{it:prop:vv:convergence} $\BF^2\llcorner A_k \to \BF^2\llcorner A$ in $\V_2(\R^3)$ as $k\to\infty$ whenever $A_k$ is a sequence of $\SL^3$-measurable sets, $A\subset\R^3$, and $\chi_{A_k}\to\chi_{A}$ in $L^1_{\mathrm{loc}}(\R^3)$ as $k\to\infty$.
	\end{enumerate}
\end{prop}

\begin{remark}\label{rem:prop:vv}\leavevmode
	\begin{enumerate}[(1)]
		\item \label{it:rem:prop:vv:integral} In view of \eqref{it:prop:vv:rectifiability}, Hypothesis \eqref{it:def:vv:integral}, and \Cref{rem:def:vv}\eqref{it:rem:def:vv:rectifiability}, $V$ is an oriented \emph{integral varifold}, i.e.\ a member of $\mathbb{IV}_2^\uo(\R^3)$ in the sense of Hutchinson \cite[Chapter 3]{Hutchinson86Indiana}. 
		\item \label{it:rem:prop:vv:density} Statements \eqref{it:prop:vv:support}\eqref{it:prop:vv:density-W}\eqref{it:prop:vv:density-V} in particular mean that the densities (i.e.\ the limit \eqref{eq:density}) exist for all points.
		\item According to \cite[Theorem 3.5(1)(a)]{Allard}, \eqref{it:prop:vv:density-W}\eqref{it:prop:vv:density-V} imply that neither $W$ nor $V_\alpha$ are rectifiable. In fact, the rectifiable part of $V_\alpha$ is precisely given by $V$, cf.\ \cite[Theorem 5.5(1)]{Allard}.
	\end{enumerate}
\end{remark}

\begin{proof}[Proof of \Cref{prop:vv}]
	Denote with $\{e_1,e_2,e_3\}$ the standard basis of $\R^3$ and with $\{\ud x_1,\ud x_2,\ud x_3\}$ its dual basis. 
	Then 
	$\omega \in C^\infty_c(\R^3,\textstyle\bigwedge^m \R^3)$
	if and only if there exists $X=(X^1,X^2,X^3)\in C_c^\infty(\R^3,\R^3)$ such that 
	\begin{equation*}
		\omega = X^1\ud x_2 \wedge \ud x_3 - X^2\ud x_1 \wedge \ud x_3 + X^3\ud x_1 \wedge \ud x_2.
	\end{equation*}
	Moreover, $\xi \in \G^\uo(3,2)$ if and only if there exists $(\xi^1,\xi^2,\xi^2)\in\mathbb\{x\in\R^3\colon|x|=1\}$ such that $\star\xi = (\xi^1,\xi^2,\xi^3)$. It follows 
	\begin{equation*}
		\xi = \xi^1 e_2 \wedge e_3 - \xi^2 e_1 \wedge e_3 + \xi^3 e_1 \wedge e_2
	\end{equation*}
	and 
	\begin{equation*}
		\ud \omega = (\Div X) \ud x_1\wedge \ud x_2 \wedge \ud x_3,\qquad \langle X,\star\xi\rangle = \omega(\xi).
	\end{equation*}
	Hence, Hypothesis \eqref{it:def:vv:divergence} and the definition of $\mathbf c(V)$ (see \eqref{eq:induced_current}) imply
	\begin{align}\label{eq:prop:vv:current}
		\partial(\BE^3\llcorner E)(\omega) & = (\BE^3\llcorner E)(\ud \omega) = \int_E\Div X\,\ud\SL^3 \\ \nonumber
		&= -\int_{\G_2^\uo(\R^3)}\langle X,\star\xi\rangle\,\ud V(x,\xi) = -\mathbf c(V)(\omega).
	\end{align}
	In particular, $\partial(\BE^3\llcorner E)(\omega) \le \int_{\R^3}\|\omega\|\,\ud\mu_V$, 
	and thus $\|\partial(\BE^3\llcorner E)\|(\R^3)\le\mu_V(\R^3)$.
	
	In view of Equation \eqref{eq:prop:vv:current} and Hypothesis \eqref{it:def:vv:boundedness}, \eqref{it:prop:vv:iso-ineq} follows from the isoperimetric inequality for sets of finite perimeter \cite[Theorem 5.6.1(i)]{EvansGariepy} and the inequality in \eqref{it:prop:vv:perimeter}.
	
	Statements \eqref{it:prop:W_by_H} and \eqref{it:prop:H_by_W} follow from the Cauchy--Schwarz inequality and Young's inequality for products.
		
	Since the continuous functions are dense in the Lebesgue spaces, Hypothesis \eqref{it:def:vv:integral} leads to
	\begin{equation*}
		V(\{(x,\xi)\in A\times\G^\uo(3,2)\colon \xi = (-1)^{i+1}\nu(x)\}) = \int_{A}\theta_i\,\ud\SH^2
	\end{equation*}
	for $i=1,2$ and all Borel sets $A\subset\spt\mu_V$. In particular, $\theta_i\ge0$. By Hypotheses \eqref{it:def:vv:mass}\eqref{it:def:vv:mc}\eqref{it:def:vv:density}, and \cite[Theorem 5.5(1)]{Allard} imply that $\bq_\# V$ is rectifiable. Thus, by \cite[Theorem 3.5(1)(b)]{Allard},
	\begin{equation}\label{eq:prop:vv:weight}
		\mu_V(\varphi) = \int_{\spt\mu_V}\varphi(x)\Theta^2(\mu_V,x)\,\ud\SH^2x\qquad \text{for all continuous $\varphi\colon\R^3\to\R$}
	\end{equation}
	which by Hypothesis \eqref{it:def:vv:integral} yields the second part of \eqref{it:prop:vv:integral}.
	
	Statement \eqref{it:prop:vv:integral} in particular implies that $\Theta^2(\mu_V,x)$ is a positive integer for $\mu_V$-almost all $x\in\R^3$. Therefore, \cite[Theorem 3.5(1)(c)]{Allard} implies that $\bq_\# V$ is integral and \eqref{it:prop:vv:perpendicular} follows from \cite[Section 5.8]{Brakke}.
	
	Now, the first statement in \eqref{it:prop:vv:support} follows for instance from \cite[Theorem 3.6]{ScharrerAGAG}. Then, the second statement is a consequence of Hypothesis \eqref{it:def:vv:mass} and \cite[2.8(2)(b)]{Allard}. 
	
	Let $C$ be a connected component of $\spt\mu_V$ and $Z\vcentcolon=(\bq_\#V)\llcorner C\times\G(3,2)$. Then $\mu_Z=\mu_V\llcorner C$ and, since $C$ is relatively open in $\spt\mu_V$, Hypotheses \eqref{it:def:vv:mass}\eqref{it:def:vv:density} are obviously satisfied by $Z$. As a consequence of \cite[Corollary 6.14]{MenneIndiana} and \cite[Lemma 5.1]{MS23a}, $Z$ also satisfies Hypothesis \eqref{it:def:vv:mc}.
	The last part of Statement \eqref{it:prop:vv:indecomposability} is then a combination of \cite[2.5]{Menne09} and \cite[Remark 6.3, Corollary 10.19]{MS23a}.
	
	By Hypotheses \eqref{it:def:vv:mass}\eqref{it:def:vv:mc} and \eqref{it:prop:vv:perpendicular}, the first statement follows from \cite[Theorem 4.2]{RuppScharrer22} applied for $c_0=0$. One may use \eqref{it:prop:vv:indecomposability} and \cite[Theorem 7.4]{MS17} to deduce that all connected components of $\spt\mu_V$ are compact. Thus,
	\begin{equation*}
		\CW(V)=\sum\Bigl\{\CW(V\llcorner C\times\G^\uo(3,2))\colon \text{$C$ connected component of $\spt\mu_V$}\Bigr\}
	\end{equation*}
	which by \eqref{it:prop:vv:indecomposability} and \cite[Theorem 4.2]{RuppScharrer22} for $c_0=0$ implies that $\spt\mu_V$ has only finitely many components, thus is compact. Hence, \eqref{it:prop:vv:rectifiability} follows from \cite[Theorems 2.8(5),\,3.5(1)(a)]{Allard}.
	
	The first part of \eqref{it:prop:vv:density-W} is clear by \Cref{def:functionals}\eqref{it:def:functionals:F}. The second part follows from the fact that $\mu_W(\bar B_r(x))\le 3\SL^3(\bar B_r(x))/2 = 2\pi r^3$ for all $r>0$. 
	
	Then \eqref{it:prop:vv:density-V} follows from \eqref{it:prop:vv:density-W} and the first part of \eqref{it:prop:vv:support} since $\Theta^2(\mu_V,x)=0$ for all $x\in\R^3\setminus\spt\mu_V$.
		
	By \eqref{it:prop:vv:perimeter}, one may combine the divergence theorem \cite[4.5.6(5)]{Federer} with Hypotheses \eqref{it:def:vv:integral}\eqref{it:def:vv:divergence} to infer
	\begin{equation}\label{eq:prop:vv:divergence}
		\int_{\R^3}\langle X,\bn_E\rangle\,\ud\SH^2 = \int_{\R^3}\langle X(x),\star\xi\rangle\,\ud V(x,\xi) = \int_{\spt\mu_V}\langle X, \star\nu\rangle (\theta_1-\theta_2)\,\ud \SH^2 
	\end{equation}
	for all $X\in C^1(\R^3,\R^3)$ which readily implies \eqref{it:prop:vv:normal} and then, in view of \eqref{it:prop:vv:integral}, also \eqref{it:prop:vv:even}\eqref{it:prop:vv:odd}.
	
	Write $S_E\vcentcolon=\spt\SL^3\llcorner E$ and $B\vcentcolon=\{x\in\R^3\colon \bn_E(x)\neq 0\}$. 
	By \cite[Proposition 3.1]{Giusti}, one may assume that
	\begin{equation}\label{eq:prop:vv:support_boundary}
		0<\SL^3(E\cap \bar B_r(x))<\SL^3(\bar B_r(x))\qquad \text{for all $x\in\partial E$ and all $r>0$}.
	\end{equation}
	In particular, $S_E=\overline E$ and $\partial S_E \subset \partial E$. Let $x\in\R^3\setminus\spt\SH^2\llcorner B$. Then there exists $r>0$ such that $\bar B_r(x)\subset\R^3\setminus\spt\SH^2\llcorner B$ and the divergence theorem combined with Equation \eqref{eq:prop:vv:current} and \cite[Corollary 4.5.3]{Federer} imply that either $\SL^3(\bar B_r(x)\cap E) = 0$ or $\SL^3(\bar B_r(x)\setminus E)=0$. Therefore, $x\in\R^3\setminus \partial E$ by Equation \eqref{eq:prop:vv:support_boundary}. In other words, $\partial S_E\subset\partial E\subset\spt\SH^2\llcorner B$. 
	Moreover, Equation \eqref{eq:prop:vv:divergence} implies $\spt\SH^2\llcorner B\subset \spt\mu_V$, so \eqref{it:prop:vv:boundary} is proved.
	
	Now, it follows from \eqref{it:prop:vv:support} that $\SL^3(\partial E)=0$ which
	implies \eqref{it:prop:vv:restriction}.
	
	Using that $\spt\mu_V$ is compact, \eqref{it:prop:vv:boundary} first implies that $\partial S_E$ is bounded. Thus, by Hypothesis \eqref{it:def:vv:boundedness},	$S_E$ itself is bounded which means $\diam S_E=\diam \partial S_E$, and \eqref{it:prop:vv:diam} follows from \eqref{it:prop:vv:boundary}.
	
	Statement \eqref{it:prop:vv:F} readily follows from the definition of the first variation \eqref{eq:first_variation} and the \Cref{def:functionals}\eqref{it:def:functionals:F} of $W = \BF^2\llcorner E$. 
	
	Then, Hypotheses \eqref{it:def:vv:mc}\eqref{it:def:vv:divergence} imply \eqref{it:prop:vv:first_variation}.
		
	Using \eqref{it:prop:vv:integral}\eqref{it:prop:vv:first_variation}, Hypothesis \eqref{it:def:vv:integral}, and \cite[3.5(1)(b)]{Allard}, one verifies  
	\begin{equation*}
		\updelta V_\alpha(X)=\int_{\spt\mu_V}\langle X, H(\theta_1+\theta_2) + \alpha(\star\nu)(\theta_1-\theta_2)\rangle \,\ud \SH^2 = \int_{\spt\mu_V}\Bigl\langle X, H+\alpha\frac{\theta_1-\theta_2}{\Theta^2(\mu_V,\cdot)}(\star\nu)\Bigr\rangle \,\ud \mu_V
	\end{equation*}
	which by \eqref{it:prop:vv:density-V} implies \eqref{it:prop:vv:mc} since $\mu_W(\spt\mu_V) = 3\SL^3(\spt\mu_V)/2=0$ by \eqref{it:prop:vv:support}.
	
	By Hypothesis \eqref{it:def:vv:integral}, there holds
	\begin{equation*}
		V\bigl(\{(x,\xi)\in\G_2^\uo(\R^3)\colon \bp(\xi)\neq\bp(\nu(x))\}\bigr) = 0.
	\end{equation*}
	Thus, \eqref{it:prop:vv:perpendicular-F} follows from \eqref{it:prop:vv:perpendicular} since $H_\alpha(x) = 0$ for all $x\in\R^3\setminus\spt\mu_V$ and $\mu_W(\spt\mu_V)=0$.
	
	To prove Statement \eqref{it:prop:vv:Helfrich-Willmore}, one applies Hypothesis \eqref{it:def:vv:integral} as well as Statement \eqref{it:prop:vv:integral} to compute
	\begin{align}\nonumber
		\CH_{-\alpha}(V) &= \frac{1}{4}\int_{\spt\mu_V}\Bigl[|H+\alpha(\star\nu)|^2\theta_1 + |H-\alpha(\star\nu)|^2\theta_2\Bigr]\,\ud\SH^2\\ \nonumber
		&=\frac{1}{4}\int_{\spt\mu_V}\Bigl[(|H|^2+\alpha^2)(\theta_1+\theta_2) + 2\alpha\langle H,\star\nu\rangle(\theta_1-\theta_2)\Bigr]\,\ud\SH^2\\ \label{eq:prop:vv:Helfrich}
		&=\frac{1}{4}\int_{\R^3}\Bigl[|H|^2 +\alpha^2 + 2\alpha\langle H,\star\nu\rangle\frac{\theta_1-\theta_2}{\theta_1+\theta_2}\Bigr]\,\ud\mu_V.
	\end{align}
	On the other hand,
	\begin{equation}\label{eq:prop:vv:singular_Willmore}
		|H_\alpha|^2= |H|^2 + 2\alpha\langle H,\star\nu\rangle \frac{\theta_1-\theta_2}{\theta_1+\theta_2} + \alpha^2\Bigl(\frac{\theta_1-\theta_2}{\theta_1+\theta_2}\Bigr)^2.
	\end{equation}
	Now, Statement \eqref{it:prop:vv:Helfrich-Willmore} is a combination of \eqref{it:prop:vv:mc} and Equations \eqref{eq:prop:vv:Helfrich}\eqref{eq:prop:vv:singular_Willmore}. 
	
	To prove \eqref{it:prop:vv:convergence}, let $\varphi\colon\G_2(\R^3)\to\R$ be a continuous function with compact support. Then, by Fubini's theorem,
	\begin{align*}
		|(\BF\llcorner A_k)(\varphi) - (\BF\llcorner A)(\varphi)| & = \frac{1}2 \Bigl|\sum_{i=1}^3\int_{\R^3}\varphi(x,E_i)(\chi_{A_k}(x) - \chi_A(x))\,\ud\SL^3x\Bigr| \\
		&\leq \|\varphi\|_{C^0}\frac{3}{2}\int_K|\chi_{A_k}-\chi_A|\,\ud\SL^3\xrightarrow{k\to\infty}0
	\end{align*}
	for the bounded set $K = \bigcup_{i=1}^3 \{x\in\R^3\colon\varphi(x,E_i)\neq0\}$.
\end{proof}

\begin{lemma}\label{lem:compactness}
	Suppose $(V_k,E_k)$ is a sequence of volume varifolds with $\spt\mu_{V_k}$ connected and $0\in E_k$ and for all $k\in\N$ such that
	\begin{equation*}
		\limsup_{k \to \infty} \bigl(\CW(V_k) + \CA(V_k)\bigr) < \infty. 
	\end{equation*}
	Then, there exists a volume varifold $(V,E)$ with $\spt\mu_V$ connected such that, after passing to a subsequence,
	\begin{gather} \label{eq:lem:compactness:convergence_current}
		\chi_{E_k}\to\chi_{E} \text{ in $L^1(\R^3)$\qquad as $k\to\infty$}, \\ \label{eq:lem:compactness:convergence_varifold}
		V_k\to V \text{ in $\V^\uo_2(\R^3)$ \qquad as $k\to\infty$}. 
	\end{gather}
	Moreover, if $\Theta^2(\mu_V,x)=1$ for $\mu_V$-almost all $x\in\R^3$, then
	\begin{equation}\label{eq:lem:compactness:lsc}
		\CH_{c_0}(V)\le\liminf_{k\to\infty} \CH_{c_0}(V_k)\qquad\text{for all $c_0\le0$}.
	\end{equation}
\end{lemma}

\begin{proof}
	Abbreviate
	\begin{equation*}
		A_*\vcentcolon=\sup_{k\in\N}\CA(V_k)<\infty,\qquad W_*\vcentcolon=\sup_{k\in\N}\CW(V_k)<\infty.
	\end{equation*}
	Since by hypothesis $0\in E_k$ for all $k\in\N$, Simon's diameter bound \cite[Lemma 1.1]{SimonCAG} (see also \Cref{lem:diam} for $\alpha=0$) and \Cref{prop:vv}\eqref{it:prop:vv:diam} provide $0<\rho<\infty$ such that
	\begin{equation}\label{eq:lem:compactness:boundedness}
		\spt \SL^3\llcorner E_k\subset B_\rho\qquad \text{for all $k\in\N$}.
	\end{equation}
	Hence, denoting $T_k\vcentcolon=\BE^3\llcorner E_k$, \Cref{prop:vv}\eqref{it:prop:vv:perimeter}\eqref{it:prop:vv:restriction} and the isoperimetric inequality \cite[Lemma 4.5.2(2)]{Federer} imply
	\begin{equation}\label{eq:lem:compactness:iso-ineq}
		\|T_k\|(\R^3) = \SL^3(E_k\cap B_\rho)\leq C\|\partial T_k\|(\R^3)^\frac{3}{2} \leq CA_*^\frac{3}{2}
	\end{equation}
	for some universal constant $C<\infty$. Thus,
	\begin{equation*}
		\|T_k\|(\R^3) + \|\partial T_k\|(\R^3)\leq CA_*^\frac{3}{2} + A_*
	\end{equation*}
	and the compactness theorem 
	for functions of bounded variation \cite[Theorem 5.2.4]{EvansGariepy} applied with $U=B_\rho$
	yields the existence of an $\SL^3$-measurable set $E$ satisfying \eqref{eq:lem:compactness:convergence_current} after passing to a subsequence.
	
	By H\"olders inequality and \Cref{prop:vv}\eqref{it:prop:vv:perimeter}, there holds
	\begin{equation*}
		\mu_{V_k}(\R^3) + \int_{\R^3}|H_k|\,\ud\mu_{V_k} + \|\partial\bc(V_k)\|(\R^3) \le A_* + 2\sqrt{A_*W_*}.
	\end{equation*}
	Hence, in view of \Cref{rem:prop:vv}\eqref{it:rem:prop:vv:integral}, the convergence in \eqref{eq:lem:compactness:convergence_varifold} follows from \cite[Theorem 3.1]{Hutchinson86Indiana}. The varifold $V$ thus obtained satisfies Hypotheses \eqref{it:def:vv:mass}\eqref{it:def:vv:density}\eqref{it:def:vv:integral} and, in view of Equation \eqref{eq:lem:compactness:boundedness}, also Hypothesis \eqref{it:def:vv:boundedness}. Combining \eqref{eq:lem:compactness:convergence_current} and \eqref{eq:lem:compactness:convergence_varifold}, the pair $(V,E)$ also satisfies Hypothesis \eqref{it:def:vv:divergence}.
	
	The definition of the first variation (see Equation \eqref{eq:first_variation}), the varifold convergence \eqref{eq:lem:compactness:convergence_varifold}, and H\"older's inequality imply
	\begin{align}\label{eq:lem:compactness:lsc_Willmore}
		\updelta V(X) = \lim_{k\to\infty}\updelta V_k(X) \le \liminf_{k\to\infty} \|X\|_{L^2(\mu_{V_k})}\|H_k\|_{L^2(\mu_{V_k})} = \|X\|_{L^2(\mu_V)}\liminf_{k\to\infty}\sqrt{4\CW(V_k)}. 
	\end{align}
	In particular, by the definition of the Willmore energy (see \Cref{def:functionals}\eqref{it:def:functionals:Willmore}),
	\begin{equation}\label{eq:lem:compactness:Willmore}
		\CW(V) \le \liminf_{k\to\infty} \CW(V_k).
	\end{equation}
	Therefore, in view of \Cref{rem:def:functionals}, $V$ also satisfies Hypothesis \eqref{it:def:vv:mc}.
	
	Next, define the sequence of Radon measures
	\begin{equation*}
		\psi_k(A)\vcentcolon=\int_A|H_k|^2\,\ud\mu_{V_k}\qquad\text{for all Borel sets $A\subset\R^3$}.
	\end{equation*}
	After passing to a subsequence, we may assume that the sequence $\psi_k$ converges to a Radon measure~$\psi$ (in the sense of \eqref{eq:weak_convergence}) satisfying $\psi(\R^3)<\infty$. For the sake of a contradiction, assume that $\spt\mu_V$ was disconnected. Hence, there existed disjoint nonempty compact sets $C_1,C_2\subset \R^3$ with $\spt\mu_V=C_1\cup C_2$. Let $\boldsymbol\gamma(2)$ be the isoperimetric constant according to \cite[Definition 3.7]{MS17}. Since $\spt\mu_{V_k}$ are connected, the varifold convergence \eqref{eq:lem:compactness:convergence_varifold} and the fact that $\psi(\R^3)<\infty$ implied the existence of $x_0\in\R^3$ and $r>0$ such that $\bar B_{2r}(x_0)\subset\R^3\setminus(C_1\cup C_2)$ and
	\begin{equation*}
		\psi(\bar B_{2r}(x_0))<\frac{1}{(2\boldsymbol\gamma(2))^2},\qquad \spt\mu_{V_k}\cap\bar B_r(x_0)\ne\varnothing\quad\text{for infinitely many $k\in\N$}. 
	\end{equation*}
	Hence, after passing to a subsequence, there holds
	\begin{equation*}
		\psi_k(\bar B_{2r}(x_0))\le\frac{1}{(2\boldsymbol\gamma(2))^2},\quad \spt\mu_{V_k}\cap\bar B_r(x_0)\ne\varnothing\qquad\text{for all $k\in\N$}. 
	\end{equation*}
	Now, \cite[2.5]{Menne09} applied with $\eps=1/(2\boldsymbol\gamma(2))$ implied $\mu_{V_k}(\bar B_{2r}(x_0))\ge r^2/(4\boldsymbol\gamma(2))^2$ for all $k\in\N$, a contradiction. 
	
	Now assume $\Theta^2(\mu_V,x)=1$ for $\mu_V$-almost all $x\in\R^3$, let $c_0\le0$ and define $\alpha\vcentcolon= - c_0$ as well as $V_{k,\alpha}\vcentcolon=V_k + \alpha\BF^2\llcorner E_k$, and $V_{\alpha}\vcentcolon=V + \alpha\BF^2\llcorner E$. Let $\theta_1,\theta_2$ be the integer valued functions associated with $V$ as in \Cref{def:vv}\eqref{it:def:vv:integral}. Then \Cref{prop:vv}\eqref{it:prop:vv:integral}  
	implies $|\theta_1(x)-\theta_2(x)|=1$ for $\mu_V$-almost all $x\in\R^3$. Hence, by \Cref{prop:vv}\eqref{it:prop:vv:Helfrich-Willmore}, $\CW(V_\alpha) = \CH_{-\alpha}(V)$. Moreover, \Cref{prop:vv}\eqref{it:prop:vv:convergence} and \eqref{eq:lem:compactness:convergence_current}\eqref{eq:lem:compactness:convergence_varifold} imply
	\begin{equation*}\label{eq:lem:compactness:convergence-F}
		V_{k,\alpha}\to V_\alpha \text{ in $\V_2(\R^3)$}\qquad \text{as $k\to\infty$}.
	\end{equation*}
	Therefore, 
	\eqref{eq:lem:compactness:lsc} follows from the lower semi-continuity of the Willmore functional \eqref{eq:lem:compactness:Willmore}. 
\end{proof}

\section{Example surfaces}

In this section example surfaces will be constructed in order to show that the Helfrich functional is not lower semi-continuous on the space of smoothly immersed surfaces with respect to oriented varifold convergence (Examples \ref{ex:lsc_negative}, \ref{ex:lsc_positive}). Moreover, it will be shown that the right hand side in \Cref{T:diam} only bounds the extrinsic but not the intrinsic diameter (\Cref{ex:diam}). The examples are made out of rotationally symmetric constant mean curvature surfaces including pieces of catenoids (\Cref{lem:catenoid_piece}), pieces of spheres (\Cref{lem:shperical_cap}), and pieces of nodoids (\Cref{lem:nodoid_piece}). Once a basic gluing method is provided (\Cref{lem:cut_sphere}), one also readily constructs examples of surfaces converging to the unit sphere whose total mean curvature diverge to $\mp\infty$ (Examples \ref{ex:total_mc-}, \ref{ex:total_mc+}).

Throughout this section, only $C^{1,1}$-regular immersed surfaces $\Sigma\subset\R^3$ that are also volume varifolds with a clear choice of inner normal are considered. The Helfrich energy of the varifolds induced by these surfaces is denoted with $\CH_{c_0}(\Sigma)$, the inner normal with $\nu$, the trace of the second fundamental form with $H$, and the induced measure with $\mu$. The area of a surface $\Sigma$ is denoted with $|\Sigma|$.

\begin{lemma}[Piece of catenoid]
	\label{lem:catenoid_piece}
	Let $a,t>0$ and denote with $\Sigma_{a,t}$ the surface that results by rotating the graph of the function
	\begin{equation*}
		f\colon [0,t) \to \mathbb R, \qquad f(x) = a \cosh\Bigl(\frac{x}{a}\Bigr).
	\end{equation*}  
	Then the following hold.
	\begin{enumerate}[\upshape(1)]
		\item \label{it:lem:catenoid_piece} $|\Sigma_{a,t}|=\pi a\bigl(\frac{a}{2}\sinh\bigl(\frac{2t}{a}\bigr) + t\bigr)$;\medskip
		\item \label{it:lem:catenoid_piece:mc} $\langle H,\nu\rangle = 0$.
	\end{enumerate}
\end{lemma}

\begin{proof}
	The first statement is readily checked using the formulas 
	\begin{equation*}
		|\Sigma_{a,t}|= 2\pi\int_0^tf\sqrt{1+\dot f^2}\,\ud x,\qquad \int \cosh\Bigl(\frac{x}{a}\Bigr)^2\,\ud x = \frac{a}{4}\sinh\Bigl(\frac{2x}{a}\Bigr) + \frac{x}{2} + C.
	\end{equation*}
	The second statement holds since $\Sigma_{a,t}$ has constant zero mean curvature (see for instance \cite[Section 3]{BenditoBowickMedina}).
\end{proof}

\begin{lemma}[Spherical cap]\label{lem:shperical_cap}
	Let $C$ be a spherical cap with polar angle $\theta_0$, and radius $r$. That is
	\begin{equation*}
		C=\{r(\sin(\theta)\cos(\varphi),\sin(\theta)\sin(\varphi),\cos(\theta))\colon 0\le\theta<\theta_0,\,0\le\varphi\le2\pi\}.
	\end{equation*}
	Then $|C| = 2\pi r h$ where $h \vcentcolon= r(1 - \cos\theta_0)$ is the \emph{height}.
\end{lemma}

\begin{lemma}[Cut sphere] \label{lem:cut_sphere}
	Let $a,t>0$ and $\Sigma_{a,t}$ be defined as in Lemma~\ref{lem:catenoid_piece}. Then there exist a unique radius $r>a$ and a height $h>0$ such that the surface $S_{a,t}$ that results by removing a spherical cap of height $h$ from a sphere of radius $r$, fits $C^{1,1}$-regular into the surface $\Sigma_{a,t}$. That is $\Sigma_{a,t}\cup S_{a,t}$ is a $C^{1,1}$-regular rotationally symmetric surface whose boundary is a circle of radius $a$. The radius of the base cap is $\rho\vcentcolon=r\sin\theta_0$ for the polar angle $\theta_0$ as in \Cref{lem:shperical_cap}.
	
	The following hold.
	\begin{enumerate}[\upshape(1)]
		\item \label{it:lem:cut_sphere:radius} $r = a \cosh(t/a)^2$;\medskip 
		\item $h = a \cosh(t/a) e^{-t/a}$;\medskip
		\item $\rho = a\cosh(t/a)$;\medskip
		\item \label{it:lem:cut_sphere:area} $|S_{a,t}| = 2\pi a^2\cosh(t/a)^3e^{t/a}$;\medskip
		\item \label{it:lem:cut_sphere:mc} $\int_{S_{a,t}}\langle H,\nu\rangle\,\ud\mu = 4\pi a\cosh(t/a)e^{t/a}$.
	\end{enumerate}
\end{lemma}

\begin{proof}
	Let $f$ be as in Lemma~\ref{lem:catenoid_piece}. In order to determine the centre $p>0$ of $S_{a,t}$ on its rotation axis, consider the system of linear equations 
	\begin{equation*}
		\begin{pmatrix}
			t \\ f(t)
		\end{pmatrix} + \lambda
		\begin{pmatrix}
			f'(t) \\ -1
		\end{pmatrix} = 	
		\begin{pmatrix}
			p \\ 0
		\end{pmatrix}.	
	\end{equation*}
	Its solution is given by 
	\begin{equation}\label{eq:lem:cut_sphere:length}
		\lambda = f(t), \qquad p = t + f(t)f'(t) = t + a \cosh(t/a)\sinh(t/a).
	\end{equation}
	It follows
	\begin{equation*}
		r = |\lambda (f'(t),-1)| = f(t)\sqrt{1 + f'(t)^2} = a\cosh(t/a)^2
	\end{equation*}
	and 
	\begin{equation*}
		h = r - (p - t) = a\cosh(t/a)\bigl(\cosh(t/a) - \sinh(t/a)\bigr) = a\cosh(t/a)e^{-t/a}.
	\end{equation*}
	Using \Cref{lem:shperical_cap} the area can be computed as 
	\begin{align*}
		|S_{a,t}| & = 4\pi r^2 - 2\pi r h = 2\pi \bigl(2a^2\cosh(t/a)^4 - a\cosh(t/a)^2a\cosh(t/a)e^{-t/a}\bigr) \\
		&=2\pi a^2 \cosh(t/a)^3\bigl(2\cosh(t/a) - e^{-t/a}\bigr) = 2\pi a^2 \cosh(t/a)^3e^{t/a}
	\end{align*}
	which implies the conclusion.
\end{proof}

\begin{example}[\protect{\cite[Theorem 1.2]{DalphinHenrotMasnouTakahashi}}]
	\label{ex:total_mc-}
	There exists a sequence of embedded spheres $\Sigma_k\subset\R^3$ converging to the unit sphere in the Hausdorff distance and as oriented varifolds (i.e. in $\V_2^\uo(\R^3)$) such that
	\begin{equation*}
		\lim_{k\to\infty}\int_{\Sigma_k}\langle H_{k},\nu_{k}\rangle\,\ud\mu_{k} = - \infty. 
	\end{equation*}
\end{example}

\begin{example}
	\label{ex:total_mc+}
	There exists a sequence of embedded spheres $\Sigma_k\subset\R^3$ converging to the unit sphere in the Hausdorff distance and as oriented varifolds such that
	\begin{equation*}
		\lim_{k\to\infty}\int_{\Sigma_k}\langle H_{k},\nu_{k}\rangle\,\ud\mu_{k} = \infty.
	\end{equation*}
\end{example}

\begin{proof}
	For all $a>0$ let $N_a$ be a number to be picked later, small enough to find $v_1,\ldots,v_{N_a}\in\Sph^2\subset \R^3$ such that the balls $B_{2\sqrt{a}}(v_1),\ldots B_{2\sqrt{a}}(v_{N_a})$ are pairwise disjoint. Let 
	\begin{equation*}
		 t_a\vcentcolon=a\arcosh 1/\sqrt{a},\qquad s_a\vcentcolon=a\arcosh 1/\sqrt[4]{a},
	\end{equation*}
	and denote with $\Sigma_a$ the spherical surface that results from perturbing $\Sph^2$ as follows. Within each ball $B_{2\sqrt{a}}(v_i)$, remove a spherical cap of height 
	\begin{equation}\label{eq:lem:total_mc+:height}
		h_a\vcentcolon=a \cosh(t_a/a) e^{-t_a/a} = \frac{\sqrt{a}}{\sqrt{\frac{1}{a}}+\sqrt{\frac{1}{a} - 1}} =\CO(a) \qquad \text{as $a\to0+$}
	\end{equation}
	and glue in the surfaces $\Sigma_{a,t_{a}},\Sigma_{a,s_a},S_{a,s_a}$ as defined in Lemmas \ref{lem:catenoid_piece} and \ref{lem:cut_sphere}. Using the formulas 
	\begin{equation*}
		\arcosh x = \ln(x + \sqrt{x^2-1}),\qquad 2\arcosh x = \arcosh(2x^2-1),\qquad \sinh \arcosh x = \sqrt{x^2-1}\qquad 
	\end{equation*}
	for $x\ge1$, as well as \Cref{lem:catenoid_piece}\eqref{it:lem:catenoid_piece}, \Cref{lem:shperical_cap}, and \Cref{lem:cut_sphere}\eqref{it:lem:cut_sphere:area}, one infers from \eqref{eq:lem:total_mc+:height} that
	\begin{align*}
		|\Sph^2\setminus S_{a,t_a}|&=2\pi h_a=\CO(a),\\
		|\Sigma_{a,t_a}|&=\frac{\pi}{2}a^2\sinh\Bigl(\frac{2t_a}{a}\Bigr) + \pi at_a = \frac{\pi}{2}a^2\sqrt{\Bigl(\frac{2}{a} - 1\Bigr)^2 - 1} +a^2\arcosh\frac{1}{\sqrt{a}} = \CO(a),\\
		|\Sigma_{a,s_a}|&=\frac{\pi}{2}a^2\sinh\Bigl(\frac{2s_a}{a}\Bigr) + \pi as_a = \frac{\pi}{2}a^2\sqrt{\Bigl(\frac{2}{\sqrt a} - 1\Bigr)^2 - 1} +a^2\arcosh\frac{1}{\sqrt[4]{a}} = \CO(a),\\
		|S_{a,s_a}| & = 2\pi a^2\cosh(s_a/a)^3e^{s_a/a} = 2\pi a^2 \Bigl(\frac{1}{\sqrt[4]{a}}\Bigr)^3\Bigl(\frac{1}{\sqrt[4]{a}} + \sqrt{\frac{1}{\sqrt{a}}-1} \Bigr)	= \CO(a)
	\end{align*}
	as $a\to0+$. It follows
	\begin{equation*}
		|\Sigma_a| = 4\pi + N_a(-|\Sph^2\setminus S_{a,t_a}|+|\Sigma_{a,t_a}|+|\Sigma_{a,s_a}|+|S_{a,s_a}|) = 4\pi + N_a\CO(a) \qquad\text{as $a\to0$}.
	\end{equation*}
	Using \Cref{lem:catenoid_piece}\eqref{it:lem:cut_sphere:mc}, one computes
	\begin{equation*}
		\int_{S_{a,s_a}}\langle H_a,\nu_a\rangle\,\ud\mu_a = 4\pi a\cosh(s_a/a)e^{s_a/a} = 4\pi a \frac{1}{\sqrt[4]{a}}\Bigl(\frac{1}{\sqrt[4]{a}} + \sqrt{\frac{1}{\sqrt{a}} - 1} \Bigr) = 8\pi \sqrt{a} + o(\sqrt{a})
	\end{equation*}
	It follows
	\begin{equation*}
		\int_{\Sigma_a}\langle H_k,\nu_k\rangle\,\ud\mu_k = 8\pi + N_a\bigl(-\CO(a) + 8 \pi \sqrt{a}+ o(\sqrt{a})\bigr) = 8\pi+N_a\sqrt{a}\bigl(8 \pi + o(1)\bigr) \qquad\text{as $a\to0$}.
	\end{equation*}
	Hence, the conclusion follows by picking $N_a$ to be the largest integer smaller than $1/a^{2/3}$, which is possible by \cite[Equation (23)]{DalphinHenrotMasnouTakahashi}.
\end{proof}

\begin{lemma}[\protect{Piece of nodoid \cite{BenditoBowickMedina,ScharrerNLA}}]
	\label{lem:nodoid_piece}
	Let $a,b>0$, $c\vcentcolon=\sqrt{a^2+b^2}$ and let $\Sigma_{a,b}$ be the surface given by the image of the map
	\begin{equation*}
		X(t,\theta) \vcentcolon = (f(t)\cos \theta, f(t) \sin \theta, g(t))
	\end{equation*}  
	for $t\geq0$, $\theta\in \R$, and
	\begin{align*}
		f(t)&\vcentcolon=b\frac{c\cosh(t)-a}{\sqrt{(c\cosh t)^2-a^2}}, \\
		g(t)&\vcentcolon=\int_0^t\sqrt{(c\cosh x)^2-a^2}\,\ud x - c\sinh(t)\frac{c\cosh(t)-a}{\sqrt{(c\cosh t)^2-a^2}}.
	\end{align*}
	Then the following hold.
	\begin{enumerate}[\upshape(1)]
		\item \label{it:lem:nodoid_piece:radii} $\sup f([0,\infty))=\lim_{t\to\infty}f(t)=b$,\quad $\min f([0,\infty))=f(0)=c-a$;\medskip
		\item \label{it:lem:nod_piece:height} $\lim_{t\to\infty}|g(t)-g(-t)|=2\bigl|a - c\bigl[E(\frac{a}{c}) - \frac{b^2}{c^2}K(\frac{a}{c})\bigr]\bigr|$ where $E,K$ are complete elliptic integrals;\medskip
		\item \label{it:lem:nod_piece:normal} $\nu(t,\theta) = (-b\cos\theta,-b\sin\theta,c\sinh t)/\sqrt{(c\cosh t)^2-a^2}$;\medskip
		\item \label{it:lem:nod_piece:mc} $\langle H,\nu\rangle = -\frac{1}{2a}$;\medskip
		\item \label{it:lem:nod_piece:area} $|\Sigma_{a,b}|=2\pi ab^2\int_{\frac{\pi}2}^\pi\sqrt{\frac{c+a\cos x}{(c-a\cos x)^3}}\,\ud x \le \pi^2 \frac{ab^2}{c}$.
	\end{enumerate}
\end{lemma}

\begin{proof}
	Statement \eqref{it:lem:nodoid_piece:radii} is \cite[Equation (5.5)]{ScharrerNLA}, \eqref{it:lem:nod_piece:height} follows from \cite[Equations (5.7)-(5.10)]{ScharrerNLA}, \eqref{it:lem:nod_piece:normal} and \eqref{it:lem:nod_piece:mc} are given in \cite[p.40]{BenditoBowickMedina}, and \eqref{it:lem:nod_piece:area} follows from \cite[Section 5.2]{ScharrerNLA}.
\end{proof}

\begin{example}\label{ex:diam}
	There exists a sequence of embedded spheres $\Sigma_k\subset\R^3$ and $0<C<\infty$ such that for 
	\begin{equation}\label{eq:ex:diam}
		d_k\vcentcolon=\sup_{p,q\in\Sigma_k}\inf\Bigl\{\int_0^1|\gamma'(t)|\,\ud t\colon \gamma\in C^1([0,1],\Sigma_k),\,\gamma(0)=p,\,\gamma(1)=q\Bigr\}
	\end{equation}
	there holds
	\begin{equation*}
		d_k\ge C k^2,\quad \int_{\Sigma_k}|H_k + 2k\nu_k|\,\ud\mu_k = \CO(k) \qquad\text{as $k\to\infty$}.
	\end{equation*}
\end{example}

The idea of the following proof is to construct a cylindrical surface that looks like a long winding snake inside a unit cube that consists of pieces of round spheres glued together by little catenoidal bridges. The length of the snake will be a lower bound for the diameter $d_k$. Since the snake consists of round spheres, most parts of the snake have the same constant scalar mean curvature whose sign depends on what is considered inside or outside. Building the snake inside a container sphere, the snake has mainly negative scalar mean curvature, cancelling out the spontaneous curvature $c_0=2k$.  

\begin{proof}
	For all $a>0$ let 
	\begin{align*}
		t_a&\vcentcolon=a\arcosh(1/a^{1/4}),\\
		r_a&\vcentcolon=a\bigl(\cosh(t_a/a)\bigr)^2 = \sqrt{a},\\
		h_a&\vcentcolon=a\cosh(t_a/a)e^{-t_a/a},\\
		l_a&\vcentcolon=2\bigl(t_a + a\cosh(t_a/a)\sinh(t_a/a)\bigr). 
	\end{align*}
	Using $\arcosh x = \log(x + \sqrt{x^2-1})$ for $x\ge 1$, one infers
	\begin{equation}\label{eq:ex:diam:length}
		2\sqrt{a}<l_a\le2\sqrt{a}(1+o(1))\qquad \text{as $a\to0$}. 
	\end{equation}
	Whenever $N_a^{1/3}\vcentcolon = 1/l_a$ is an integer, divide the unit cube $[0,1]^3\subset\R^3$ into cubes $Q_1,\ldots,Q_{N_a}$ of edge length $l_a$ ordered in a way that each pair $(Q_i,Q_{i+1})$ shares a face and $Q_{N_a}$ touches the boundary of $[0,1]^3$. Denote with $T_{a,t_a}$, respectively $\hat T_{a,t_a}$, the cylindrical surfaces that result from $S_{a,t_a}$ as defined in \Cref{lem:cut_sphere} by removing a second spherical cap of height $h_a$ which is opposite to the first removed cap, respectively in a right angle to the first removed cap. Each of the surfaces thus obtained can be glued together with two pieces of catenoids. Denote the glued surfaces with $\Gamma_a\vcentcolon=\Sigma_1\cup T_{a,t_a}\cup\Sigma_2$ and $\hat \Gamma_a\vcentcolon=\hat\Sigma_1\cup \hat T_{a,t_a}\cup\hat\Sigma_2$ where $\Sigma_i,\hat\Sigma_i$ for $i=1,2$ are suitable translations and rotations of the surfaces $\Sigma_{a,t_a}$ constructed in \Cref{lem:catenoid_piece}. Each of the surfaces $\Gamma_a,\hat \Gamma_a$ is of the topological type of a cylinder and, in view of \eqref{eq:lem:cut_sphere:length}, fits into any of the cubes $Q_i$ such that its two boundary components are contained in opposite, respectively neighboured faces of $Q_i$. Now, let 
	\begin{equation*}
		\Delta_a\vcentcolon= D_1 \cup C_2\cup\ldots \cup C_{N_a} 
	\end{equation*}
	a surface of disk type such that $D_1$ is a rigid motion of $S_{a,t_a}\cup\Sigma_{a,t_a}$ sitting in $Q_1$ and each $C_i$ is a rigid motion of either $\Gamma_a$ or $\hat \Gamma_a$ sitting in $Q_i$ where one of the two boundary components of $C_{N_a}$ lies in the face $\{(x_1,x_2,0)\colon x_1,x_2\in[0,1]\}$ of $[0,1]^3$. Let $\Sigma \subset \R^3\setminus (-1,2)^3$ be a spherical surface with $P\vcentcolon=\{(x_1,x_2,-1)\colon x_1,x_2\in[-1,2]\}\subset\Sigma$. In order to construct the final family of spherical surfaces $\Sigma_a$, one cuts out a flat disk $D\subset P$ of radius $R_a\vcentcolon=\sqrt{a(2+a)}$ and connects the two disk type surfaces $\Delta_a,\Sigma\setminus D$ with the surface $C_{N_a+1}\vcentcolon=\Sigma_{1,R_a}$ defined in \Cref{lem:nodoid_piece} by moving the two surfaces $\Delta_a,\Sigma\setminus D$ together along the $x_3$-axis. Using \Cref{lem:cut_sphere}\eqref{it:lem:cut_sphere:radius} and \Cref{lem:catenoid_piece}\eqref{it:lem:catenoid_piece}\eqref{it:lem:catenoid_piece:mc} one readily checks for $j\in\{2,\ldots,N_a\}$ that
	\begin{align*}
		\frac{1}{2}\int_{D_1}\Bigl|H + \frac{2\nu}{\sqrt{a}}\Bigr|\,\ud \mu=\int_{C_j}\Bigl|H + \frac{2\nu}{\sqrt{a}}\Bigr|\,\ud \mu = 4\pi\sqrt{a}\Bigl(\frac{a}{2}\sinh\Bigl(\frac{2t_a}{a}\Bigr) + t_a\Bigr) = \CO(a)\qquad \text{as $a\to0$}.
	\end{align*}
	Moreover, in view of \Cref{lem:nodoid_piece}\eqref{it:lem:nod_piece:mc}\eqref{it:lem:nod_piece:area}
	\begin{equation*}
		\int_{C_{N_a+1}}\Bigl|H + \frac{2\nu}{\sqrt{a}}\Bigr|\,\ud\mu = \CO(\sqrt{a}),\quad\int_{\Sigma\setminus D}\Bigl|H + \frac{2\nu}{\sqrt{a}}\Bigr|\,\ud\mu = \CO\Bigl(\frac{1}{\sqrt{a}}\Bigr) \qquad \text{as $a\to0$}. 
	\end{equation*}
	Counting the pieces together using \eqref{eq:ex:diam:length}, one concludes
	\begin{equation*}
		\int_{\Sigma_a}\Bigl|H + \frac{2\nu}{\sqrt{a}}\Bigr|\,\ud\mu = \CO\Bigl(\frac{1}{\sqrt{a}}\Bigr)\qquad \text{as $a\to0$}.
	\end{equation*}
	Finally, the diameter \eqref{eq:ex:diam} satisfies 
	\begin{equation*}
		d_a\ge N_a\Bigl(\frac{\pi}{2} - o(1)\Bigr)r_a \ge C\frac{1}{a} \qquad\text{as $a\to0$}
	\end{equation*}
	which concludes the proof.
\end{proof}

\begin{example}
	\label{ex:lsc_negative}
	For all $c_0<0$ there exists a sequence of smooth closed submanifolds $\Sigma_k$ of $\R^3$ converging in the Hausdorff distance and as Radon measures to the (not disjoint) union $\Sigma$ of two embedded spherical surfaces such that
	\begin{equation*}
		\lim_{k\to \infty}\CH_{c_0}(\Sigma_k)<\CH_{c_0}(\Sigma).
	\end{equation*}
\end{example}

\begin{remark}\label{rem:ex:lsc}
	In view the compactness \Cref{lem:compactness}, a subsequence of $\Sigma_k$ converges as oriented varifolds to a volume varifold $V$. The Helfrich energy of $V$ is uniquely determined by \Cref{prop:vv}\eqref{it:prop:vv:normal}. Indeed, since $\mu_V = \mu_\Sigma$, there holds $\CH_{c_0}(V)=\CH_{c_0}(\Sigma)$. This means that the functional $\CH_{c_0}$ is not lower semi-continuous with respect to oriented varifold convergence on the closure of the set of closed submanifolds. Obviously, by \eqref{eq:lem:compactness:lsc}, the varifold $V$ must have a nontrivial set of multiplicity points.
\end{remark}

\begin{proof}[Proof of \Cref{ex:lsc_negative}]
	Let $a\vcentcolon=-1/(2c_0)$ and for all $b>0$ define
	\begin{align*}
		&c_b\vcentcolon=\sqrt{a^2+b^2},
		&&r_b^\mathrm{out}\vcentcolon=b,\\
		&h_b\vcentcolon= 2a\Bigl|1 - \frac{c_b}a\Bigl[E\Bigl(\frac{a}{c_b}\Bigr) - \frac{b^2}{c_b^2}K\Bigl(\frac{a}{c_b}\Bigr)\Bigr]\Bigr|
		&&r_b^\mathrm{in}\vcentcolon=c_b-a.
	\end{align*}
	Let $D\vcentcolon=\{(x,y,0)\in\R^3\colon x^2+y^2\le1\}$, choose an embedded sphere $S\subset\{(x,y,z)\in\R^3\colon z\leq0\}$ such that $D\subset S$, let $\Sigma$ be the surface (with multiplicity) given by $S\cup(-S)$, and define
	\begin{equation*}
		S_b\vcentcolon=-S+(0,0,h_b),\qquad D_b\vcentcolon=D+(0,0,h_b).
	\end{equation*}
	Let $\Sigma_b$ be the genus $g_b\vcentcolon=\frac{1}{(2r_b^\mathrm{out})^2}-1$ surface resulting from $S\cup S_b$ by removing $(g_b+1)$ disks of radius $r_b^\mathrm{out}$ from both $D$ and $D_b$, and adding $2(g_b+1)$ rotated and translated copies of $\Sigma_{a,b}$ as defined in \Cref{lem:nodoid_piece}. Noting that the unit normal $\nu$ as defined in \Cref{lem:nodoid_piece}\eqref{it:lem:nod_piece:normal} points to the inside of the surface $\Sigma_b$, one infers from \Cref{lem:nodoid_piece}\eqref{it:lem:nod_piece:mc} that
	\begin{equation*}
		\CH_{c_0}(\Sigma_b)=\CH_{c_0}(\Sigma) - c_0^22(g_b+1)\pi (r_b^\mathrm{out})^2= \CH_{c_0}(\Sigma) - c_0^2\frac{\pi}{2}.
	\end{equation*}
	In order to see that $\Sigma_b\xrightarrow{b\to0}\Sigma$ in the Hausdorff distance, it is enough to notice $r_b^\mathrm{in}\xrightarrow{b\to0}0$ and $h_b\xrightarrow{b\to0}0$ The latter limit can be computed for instance using \cite[Equation (2.2)]{ScharrerNLA}. Moreover, from \Cref{lem:nodoid_piece}\eqref{it:lem:nod_piece:area} it follows
	\begin{equation*}
		\frac{2|\Sigma_{a,b}|}{2\pi (r_b^\mathrm{out})^2}\xrightarrow{b\to0}1
	\end{equation*}
	which implies the conclusion.
\end{proof}

\begin{example}
	\label{ex:lsc_positive}
	For all $c_0>0$ there exists a sequence of smooth closed submanifolds $\Sigma_k$ of $\R^3$ converging in the Hausdorff distance and as Radon measures to an immersed sphere $\Sigma$ such that
	\begin{equation*}
		\lim_{k\to \infty}\CH_{c_0}(\Sigma_k)<\CH_{c_0}(\Sigma).
	\end{equation*}
\end{example}

\begin{proof}
	Let $a\vcentcolon=1/(2c_0)$ and define $c_b,h_b,r_b^\mathrm{out},r_b^\mathrm{in},D,D_b$ depending on $a$ as in \Cref{ex:lsc_negative}. Choose a family of regular smooth plane curves $\gamma_h=(x_h,y_h)\colon[0,1]\to\R^2$ corresponding to $h\ge0$ such that
	\begin{align*}
		&\gamma_h(0)=(1,0), &&\dot\gamma_h(0)=(1,0), && x_h\ge1, \\
		&\gamma_h(1)=(1,h), &&\dot\gamma_h(1)=-(1,0)
	\end{align*}
	such that $\gamma_h\to\gamma_0$ smoothly as $h\to0$, and define
	\begin{equation*}
		\Gamma_b\vcentcolon=\{(x_{h_b}(t)\cos \theta,x_{h_b}(t)\sin\theta,y_{h_b}(t)\colon \theta\ge0,\,t\in[0,1])\}.
	\end{equation*}
	Let $S_b$ be the embedded sphere $\Gamma_b \cup D\cup D_b$ and $\Sigma_b$ be the genus $g_b\vcentcolon=1/(2r_b^\mathrm{out})^2$ surface resulting from $S_b$ by removing $g_b$ disks from both $D$ and $D_b$, and adding $2g_b$ rotated and translated copies of $\Sigma_{a,b}$ as defined in \Cref{lem:nodoid_piece}. Then $S_b\to S_0 =\vcentcolon\Sigma$ smoothly with multiplicity and $\Sigma_b\to\Sigma$ in Hausdorff distance and as measures as $b\to0$. Moreover, since the unit normal $\nu$ as defined in \Cref{lem:nodoid_piece}\eqref{it:lem:nod_piece:normal} points to the outside of the surface $\Sigma_b$, one infers from \Cref{lem:nodoid_piece}\eqref{it:lem:nod_piece:mc} that
	\begin{equation*}
		\CH_{c_0}(\Sigma_b)=\CH_{c_0}(S_b) - c_0^22g_b\pi (r_b^\mathrm{out})^2= \CH_{c_0}(S_b) - c_0^2\frac{\pi}{2}
	\end{equation*}
	and the conclusion follows since $\CH_{c_0}(S_b)\xrightarrow{b\to0}\CH_{c_0}(\Sigma)$.
\end{proof}

\section{Diameter bounds}\label{sec:diam}
In this section, the diameter bound \Cref{T:diam} will be proved, see \Cref{thm:diam}. The proof is based on the lower density bound \Cref{lem:lower_density-bound}. To ease the reader, a basic property concerning the density of general varifolds will be recalled first (see \Cref{lem:semi-continuity}).

\begin{lemma}\label{lem:semi-continuity}
	Suppose $R>0$, $V\in\V_2(B_R)$, $\alpha\ge0$, and 
	\begin{equation*}
		\|\updelta V\|(\bar B_r(x)) \le \alpha R \mu_V(\bar B_r(x))\qquad \text{for all $x\in B_R$ and $\SL^1$-almost all $0<r<R-|x|$}.
	\end{equation*}	
	Then the density $\Theta^2(\mu_V,x)$ exists for all $x\in B_R$, and $\limsup_{y\to x}\Theta^2(\mu_V,y)\le \Theta^2(\mu_V,x)$.
\end{lemma}

\begin{proof}
	By \cite[Theorem 17.6]{SimonGMT}, 
	the function $r\mapsto e^{\alpha R r}r^{-2}\mu_V(\bar B_r(x))$ is nondecreasing for all $x\in B_R$. In particular, the density exists for all $x\in B_R$. Given $x\in B_R$ and $\eps>0$, choose $0<r_0<R-|x|$ such that
	\begin{equation*}
		e^{\alpha Rr}\frac{\mu_V(\bar B_r(x))}{\pi r^2} \le \Theta^2(\mu_V,x) + \eps \qquad \text{for all $0<r<r_0$}.
	\end{equation*}  
	For $r,\delta>0$ with $r+\delta<r_0$, and $y\in B_\delta(x)$, it follows
	\begin{equation*}
		\Theta^2(\mu_V,y)\le e^{\alpha Rr}\frac{\mu_V(\bar B_r(y))}{\pi r^2}\le e^{\alpha R(r+\delta)}\frac{\mu_V(\bar B_{r+\delta}(x))}{\pi r^2} \le \Bigl(1+\frac{\delta}{r}\Bigr)^2(\Theta^2(\mu_V,x) + \eps)
	\end{equation*}
	which implies the conclusion. 
\end{proof}

\begin{lemma}\label{lem:lower_density-bound}
	For all $0<\eps<1$ there exists $0<\gamma<\infty$ with the following property. If $R>0$, $\alpha\ge0$, $(V,E)$ is a volume varifold, $V_\alpha\vcentcolon=\alpha(\BF^2\llcorner E) + \bq_\# V$, and 
	\begin{align} \label{eq:lem:ldb:density}
		&\Theta^2(\mu_{V_\alpha},0)\ge1, \\ \label{eq:lem:ldb:Poincare}
		&r\|\updelta V_\alpha\|(\bar B_r) \le \gamma \mu_{V_\alpha}(\bar B_r) \qquad\text{for all $0<r<R$},
	\end{align}
	then
	\begin{equation*}
		\mu_{V_\alpha}(\bar B_r)\ge\varepsilon\pi r^2\qquad\text{for all $0<r\le R$}.
	\end{equation*}
\end{lemma}

\begin{proof}
	If the lemma were false, there would exist $0<\eps<1$ and a sequence $(V_k,E_k)$ of volume varifolds such that \eqref{eq:lem:ldb:density}\eqref{eq:lem:ldb:Poincare} hold for $V_{k,\alpha}\vcentcolon=\alpha(\BF^2\llcorner E_k) + \bq_\#V_k$ and $\gamma=1/k$ but
	\begin{equation*}
		R_k\vcentcolon=\sup\{r\ge0\colon \mu_{V_{k,\alpha}}(\bar B_s)\ge \eps \pi s^2\text{ for all }0<s<r\}<R.
	\end{equation*}
	By \eqref{eq:lem:ldb:density} there holds $R_k>0$, and the definition of $R_k$ implies
	\begin{equation*}
		\mu_{V_{k,\alpha}}(\bar B_{R_k})=\eps \pi R_k^2.
	\end{equation*}
	For each $k\in\N$, denote with $\hat V_k$, respectively $\hat V_{k,\alpha}$, the varifold that results from mapping the varifold $V_k$, respectively $V_{k,\alpha}$, under the scaling $F_k(x)=R_k^{-1}x$ according to \cite[3.2]{Allard}. Then (see 3.2(2) and 4.12(1) in \cite{Allard})
	\begin{equation}\label{eq:lem:ldb:scaling}
		\mu_{\hat V_{k,\alpha}}(\bar B_s) = \frac{1}{R_k^2}\mu_{V_{k,\alpha}}(\bar B_{sR_k}),\quad \|\updelta \hat V_{k,\alpha}\|(\bar B_{s})=\frac{1}{R_k}\|\updelta V_{k,\alpha}\|(\bar B_{sR_k}) \qquad\text{for all $s>0$}.
	\end{equation}
	It follows that
	\begin{equation}\label{eq:lem:ldb:bounds}
		\mu_{\hat V_{k,\alpha}}(B_1) = \pi\eps,\qquad  \|\updelta \hat V_{k,\alpha}\|(B_1)\le \frac{\pi\eps}{k}
	\end{equation}
	and
	\begin{equation}\label{eq:lem:ldb:k}
		\mu_{\hat V_{k,\alpha}}(\bar B_s)\ge \pi \eps s^2 \qquad \text{for all $0<s\le1$}.
	\end{equation}
	By compactness of Radon measures \cite[2.6(2)(a)]{Allard}, there exists $V_\alpha$ such that, after passing to a subsequence, $\hat V_{k,\alpha}\to V_\alpha$ in $\V_2(B_1)$ as $k\to \infty$. Due to \eqref{eq:lem:ldb:bounds}, this varifold satisfies
	\begin{equation}\label{eq:lem:ldb:upper_area_bound}
		\mu_{V_\alpha}(B_1) \le \pi\eps.
	\end{equation}
	\Cref{prop:vv}\eqref{it:prop:vv:density-V}\eqref{it:prop:vv:even}\eqref{it:prop:vv:odd}\eqref{it:prop:vv:mc} imply
	\begin{equation*}
		\|\updelta V_{k}\|(\bar B_r(x))\le \|\updelta V_{k,\alpha}\|(\bar B_r(x)) + \alpha\mu_{V_k}(\bar B_r(x))\qquad \text{for all $k\in\N$, $x\in \R^3$, and $r>0$}.
	\end{equation*}
	Thus, by \eqref{eq:lem:ldb:scaling}\eqref{eq:lem:ldb:bounds},
	\begin{equation}\label{eq:lem:ldb:rectifiable_part}
		\|\updelta \hat V_k\|(\bar B_s(x)) \le \frac{\pi\eps}{k} + \alpha R \mu_{\hat V_k}(\bar B_s(x))\le \frac{\pi\eps}{k} + \alpha R \pi\eps \qquad \text{for all $x\in B_1$ and $0<s<1-|x|$}
	\end{equation}
	and by \eqref{eq:lem:ldb:k},
	\begin{equation}\label{eq:lem:ldb:ldb}
		\pi \eps s^2 \le \mu_{V_k}(\bar B_s) + \alpha \frac{3}2 R\SL^3(\bar B_s)\qquad\text{for all $0<s<1$}.
	\end{equation}
	Notice that by the scaling property \cite[3.2(2)]{Federer}, there holds $\Theta^2(\mu_{\hat V_k},\cdot)=\Theta^2(\mu_{V_k},\cdot)$. In particular, $\bq_\#\hat V_k$ are integral varifolds and the compactness \cite[Theorem 6.4]{Allard} implies that after passing to a subsequence, $\hat V_k\to V$ in $\V_2(B_1)$ as $k\to\infty$ for some $V\in \V_2(B_1)$ satisfying $\Theta^2(\mu_V,x)\ge1$ for $\mu_V$-almost all $x\in\R^3$. It follows from \eqref{eq:lem:ldb:rectifiable_part} that
	\begin{equation*}
		\|\updelta V\|(\bar B_s(x)) \le \alpha R\mu_{V}(\bar B_s(x)) 
	\end{equation*}
	for all $x\in B_1$ and all but countably many $0<s<1-|x|$. Moreover, $0\in \spt\mu_V$ by \eqref{eq:lem:ldb:ldb}. Hence, \Cref{lem:semi-continuity} implies $\Theta^2(\mu_V,0)\ge1$. Since $V_\alpha$ is stationary, it follows from the monotonicity identity \cite[(40.3)]{SimonGMT} that $\Theta^2(\mu_{V_\alpha},0)$ exists, and $\pi^{-1}\mu_{V_\alpha}(B_1)\ge \Theta^2(\mu_{V_\alpha},0) \ge \Theta^2(\mu_V,0)\ge 1$, which contradicts \eqref{eq:lem:ldb:upper_area_bound}.
\end{proof}

\begin{remark*}
	The proof is adapted from \cite[Lemma 4.1]{KasaiTonegawa}
\end{remark*}

\begin{theorem}\label{thm:diam}
	Suppose $(V,E)$ is a volume varifold, $\spt\mu_V$ is connected, and $\alpha\ge0$. Then
	\begin{equation*}
		\diam \spt \mu_V\leq C\int_{\G_2^\uo(\R^3)}|H(x)+\alpha(\star\xi)|\,\ud V(x,\xi )
	\end{equation*}
	for some universal constant $C<\infty$.
\end{theorem}

\begin{proof}
	Let $\gamma$ be the constant of \Cref{lem:lower_density-bound} for $\eps=1/2$, define $V_\alpha\vcentcolon=\alpha (\BF^2\llcorner E) + \bq_\#V$, and let
	\begin{equation*}
		S\vcentcolon=\Bigl\{x\in\spt\mu_{V}\colon \limsup_{s\to0+}\frac{s\|\updelta V_\alpha\|(\bar B_s(x))}{\mu_{V_\alpha}(\bar B_s(x))}<\gamma\Bigr\}.
	\end{equation*}
	By \cite[Theorem 2.9.5]{Federer}, there holds $\mu_V(\R^3\setminus S)=0$.
	For all $x\in S$ let
	\begin{equation*}
		r\vcentcolon=\inf\{s>0\colon s\|\updelta V_\alpha\|(\bar B_s(x))\ge\gamma\mu_{V_\alpha}(\bar B_s(x))\}.
	\end{equation*}
	One may use \cite[Theorem 5.9]{MenneScharrerKodai} to deduce $0<r<\infty$. The definitions of $S$ and $r$ then imply $r\|\updelta V_\alpha\|(\bar B_r(x)) \ge \gamma\mu_{V_\alpha}(\bar B_r(x))$ and
	\begin{equation*}
		s\|\updelta V_{\alpha}\|(\bar B_s(x))\le \gamma \mu_{V_\alpha}(\bar B_s(x))\qquad \text{for all $0<s< r$}.
	\end{equation*}
	Hence, from \Cref{prop:vv}\eqref{it:prop:vv:support}\eqref{it:prop:vv:density-V} and \Cref{lem:lower_density-bound}, it follows
	\begin{equation*}\label{eq:thm:diam:bounds}
		\frac{\pi}{2}r^2\le\mu_{V_\alpha}(\bar B_r(x)), \qquad r\le \frac{2}{\pi\gamma}\|\updelta V_{\alpha}\|(\bar B_r(x)).
	\end{equation*}
	Let $p\colon\R^3\to\R$ be an orthogonal projection satisfying $\diam p(S) = \diam \spt\mu_V$. The Besicovitch covering theorem (see for instance \cite[Theorem 2.7(1)]{Allard}) then implies 
	\begin{equation*}
		\SL^1(p(S)) \le \frac{2\mathbf B(1)}{\pi\gamma} \|\updelta V_\alpha\|(\R^3)
	\end{equation*}
	for some universal constant $\mathbf B(1)<\infty$.
	Now, the conclusion of the proof follows from \Cref{prop:vv}\eqref{it:prop:vv:indecomposability} and \cite[Theorem 7.12]{MS23a}.
\end{proof}

\begin{remark*}
	The proof follows a method of Menne \cite{MenneDiam}.
\end{remark*}

\section{Isoperimetric inequalities}\label{sec:iso}

In this section, a Sobolev inequality for volume varifolds \Cref{lem:iso} will be proved. Resulting isoperimetric inequalities are formulated in \Cref{thm:iso} and \Cref{cor:iso}. To begin with, a basic lemma used to cut out singularities will be proved (see \Cref{lem:approx}).

\begin{lemma}\label{lem:approx}
	Suppose $V$ is a volume varifold, $x_0\in\R^3$, and $X\colon\R^3\to\R^3$ is defined by $X(x)\vcentcolon=x-x_0$. Then 
	\begin{equation*}
		\liminf_{r\to0+}\frac{1}{r}\int_{\bar B_r(x_0)\times\G^\mathrm{o}(3,2)}\frac{|X^\top|^2}{|X|^3}\,\ud V \ge \pi\Theta^2(\mu_V,x_0).
	\end{equation*}
\end{lemma}

\begin{proof}
	Since $\Theta^2(\mu_V,x_0)$ exists  (\Cref{rem:prop:vv}\eqref{it:rem:prop:vv:density}), Simon's monotonicity identity \cite[Equation (1.2)]{SimonCAG} yields
	\begin{equation*}
		\lim_{r\to0+}\int_{\bar B_r(x_0)\times G^\uo(3,2)}\Bigl|\frac{1}{4}H + \frac{X^\bot}{|X|^2}\Bigr|^2\,\ud V = 0.
	\end{equation*}
	Therefore, since 
	\begin{equation*}
		\frac{|X^\bot|^2}{|X|^4}\le \frac{1}{8}|H|^2 + 2\Bigl|\frac{1}{4}H + \frac{X^\bot}{|X|^2}\Bigr|^2, 
	\end{equation*}
	one may use Hypothesis \eqref{it:def:vv:mc} to deduce
	\begin{equation*}
		\lim_{r\to0+}\int_{\bar B_r(x_0)\times\G^\mathrm{o}(3,2)}\frac{|X^\bot|^2}{|X|^4}\,\ud V = 0.
	\end{equation*}
	Now 
	\begin{equation*}
		\frac{1}{r}\int_{\bar B_r(x_0)\times\G^\mathrm{o}(3,2)}\frac{|X^\top|^2}{|X|^3}\,\ud V \ge \int_{\bar B_r(x_0)\times\G^\mathrm{o}(3,2)}\Bigl(\frac{1}{r^2} -\frac{|X^\bot|^2}{|X|^4}\Bigr)\,\ud V
	\end{equation*}
	implies the conclusion.
\end{proof}

\begin{lemma}\label{lem:iso}
	Suppose $c_0\le0$, $(V,E)$ is a volume varifold, and $0\le f\in C^1(\R^3)$. Then
	\begin{align*}
		\pi\int_{\R^3}f^2\,\ud\mu_V &\le \Bigl(\int_{\G^\uo_2(\R^3)} |\nabla f(x)| + f(x)|H(x)-c_0(\star\xi)|\,\ud V(x,\xi)\Bigr)^2 \\
		&\qquad +c_0\int_{\R^3}\int_{E}\frac{f(x)f(y)}{|x-y|^2}\,\ud\SL^3x\,\ud\mu_Vy \\
		&\qquad + c_0 \int_{\R^3}\int_E\frac{2f(y)}{|x-y|}\Bigl(|\nabla f(x)| + f(x)|H(x)-c_0(\star\xi)|\Bigr)\,\ud\SL^3y\,\ud V(x,\xi)\\
		&\qquad +c_0\int_{\R^3}\int_{E}\frac{\langle \nabla f(x),x-y\rangle}{|x-y|^2}f(y)\,\ud\SL^3x\,\ud\mu_Vy\\
		&\qquad +c_0\int_{\R^3}\int_E\frac{\langle \nabla f(y),y-x\rangle}{|y-x|}\Bigl(|\nabla f(x)| + f(x)|H(x)-c_0(\star\xi)|\Bigr)\,\ud\SL^3y\,\ud V(x,\xi).
	\end{align*}
\end{lemma}

\begin{proof}
	Let $y\in\R^3$ and $X\colon\R^3\setminus\{y\}\to\R^3$ be defined by $X(x)\vcentcolon=\frac{x-y}{|x-y|^2}$. Moreover, let $\eta\in C^1(\R^3)$ such that $y\notin\spt \eta$. Then $(\eta f X)\in C^1(\R^3,\R^3)$ and the formula 
	\begin{equation*}
		\mathrm DX_x=\frac{1}{|x-y|^2}\Bigl(\Id-2\frac{x-y}{|x-y|}\otimes\frac{x-y}{|x-y|}\Bigr) 
	\end{equation*}
	together with the first variation identity (\eqref{eq:first_variation}, \Cref{def:vv}\eqref{it:def:vv:mc}) imply
	\begin{equation*}
		-\int_{\R^3} f\langle X,H\rangle\eta\,\ud\mu_V=\int\bigl[\langle (\nabla f)^\top,X\rangle + 2f|X^\bot|^2\bigr]\eta + f\langle(\nabla\eta)^\top,X\rangle\,\ud V. 
	\end{equation*}
	A suitable choice of $\eta$ such that $\eta(x)=1$ for $|x-y|>r$ in conjunction with \Cref{lem:approx} for $r\to0+$ implies
	\begin{align}
		\pi\Theta^2(\mu_V,y)f(y) &\le \int_{\R^3}\frac{|\nabla f(x)|}{|x-y|}-f\langle X(x),H(x)\rangle\,\ud\mu_Vx \nonumber \\ \label{eq:lem:iso:y}
		&\le \int_{\G^\uo_2(\R^3)}\frac{|\nabla f(x)|}{|x-y|}+f(x)\frac{|H(x) - c_0(\star\xi)|}{|x-y|} - c_0f(x)\langle X(x),\star\xi\rangle\,\ud V(x,\xi). 
	\end{align}

	Now let $x\in\R^3$ and $Y\colon\R^3\setminus\{0\}\to\R^3$ be defined by $Y(y)\vcentcolon=\frac{y-x}{|y-x|}$. Then, similarly,
	\begin{equation*}
		\mathrm DY_y = \frac{1}{|y-x|}\Bigl(\Id - \frac{y-x}{|y-x|}\otimes\frac{y-x}{|y-x|}\Bigr)
	\end{equation*}
	and
	\begin{equation*}
		-\int_{\R^3}f\langle Y,H\rangle\,\ud\mu_V = \int_{\G^\uo_2(\R^3)} \langle (\nabla f)^\top(y,\upsilon), Y(y)\rangle +\frac{f(y)}{|y-x|}(1+|Y^\bot(y,\upsilon)|^2)\,\ud V(y,\upsilon)
	\end{equation*}
	which implies
	\begin{align*}
		\int_{\R^3}\frac{f(y)}{|y-x|}\,\ud\mu_Vy&\le\int |\nabla f(y)| + f(y)|H(y)-c_0(\star\upsilon)| - c_0f(y)\langle Y(y),\star\upsilon\rangle\,\ud V(y,\upsilon). 
	\end{align*}
	Thus, multiplying \eqref{eq:lem:iso:y} with $f(y)$ and integrating with respect to $y$ in conjunction with Fubini's theorem yields
	\begin{align*}
		\pi\int_{\R^3}f^2\,\ud\mu_V&\le\int_{\G^\uo_2(\R^3)}\int_{\R^3}\frac{f(y)}{|x-y|}\,\ud\mu_Vy \Bigl(|\nabla f(x)| + f(x)|H(x)-c_0(\star\xi)|\Bigr)\,\ud V(x,\xi) + R_1 \\
		& \le \Bigl(\int_{\G^\uo_2(\R^3)} |\nabla f(x)| + f(x)|H(x)-c_0(\xi)|\,\ud V(x,\xi)\Bigr)^2 + R_1 + R_2,
	\end{align*} 
	where the remaining terms 
	\begin{align*}
		R_1&\vcentcolon=-c_0\int_{\R^3}\int_{\G^\uo_2(\R^3)}f(x)\frac{\langle x-y,\star\xi\rangle}{|x-y|^2}\,\ud V(x,\xi)f(y)\,\ud\mu_Vy \\
		R_2&\vcentcolon=-c_0\int_{\G^\uo_2(\R^3)}\int_{\G^\uo_2(\R^3)}f(y)\frac{\langle y-x,\star\upsilon\rangle}{|y-x|}\Bigl(|\nabla f(x)| + f(x)|H(x)-c_0(\star\xi)|\Bigr)\,\ud V(y,\upsilon)\,\ud V(x,\xi)
	\end{align*}
	can be computed using \Cref{def:vv}\eqref{it:def:vv:divergence}.
\end{proof}

\begin{remark*}
	The proof is adapted from \cite[Section 4]{Topping08}.
\end{remark*}

\begin{theorem}\label{thm:iso}
	Suppose $c_0\le0$ and $(V,E)$ is a volume varifold. Then
	\begin{align*}
		\pi\CA(V)&\le \Bigl(\int_{\G^\uo_2(\R^3)}|H(x)-c_0(\star\xi)|\,\ud V(x,\xi)\Bigr)^2 +c_0\int_{\R^3}\int_{E}\frac{1}{|x-y|^2}\,\ud\SL^3x\,\ud\mu_V y \\
		&\qquad + 2c_0 \int_{\R^3}\int_E\frac{|H(x)-c_0(\star\xi)|}{|x-y|}\,\ud\SL^3y\,\ud V(x,\xi).\\
	\end{align*}
\end{theorem}

\begin{cor}\label{cor:iso}
	\begin{equation*}
		\pi \CA(V) + 2|c_0|C_{\ref{thm:diam}}^{-1}\CV(V) \le \Bigl(\int_{\G^\uo_2(\R^3)}|H(x)-c_0(\star\xi)|\,\ud V(x,\xi)\Bigr)^2 +c_0\int_{\R^3}\int_{E}\frac{1}{|x-y|^2}\,\ud\SL^3x\,\ud\mu_V(y).
	\end{equation*}
\end{cor}

\section{Lower semi-continuity}\label{sec:lsc}

In this section, lower semi-continuity of the Helfrich functional will be proved in \Cref{thm:lsc}. To be more precise, it will be shown that
\begin{equation*}
	\CH_{c_0}(V)\le\liminf_{k\to \infty}\CH_{c_0}(V_k)
\end{equation*}
provided $V_k$ converges to $V$ as volume varifolds, and $V_k$ is a sequence in $\QV_{2,C}(\R^3)$, a class containing all varifolds that arise as oriented varifold limit of a sequence of smoothly embedded surfaces whose second fundamental forms are bounded with respect to the $L^2$-norm by $C$. The space $\QV_{2,C}(\R^3)$ will be defined in \Cref{def:qev}. The proof is of local nature concentrated near points $x_0\in\R^3$ where the \emph{approximate Tangent plane} $P(x_0)=\bp(\nu(x_0))=\bp(-\nu(x_0))$ (cf.\ \Cref{rem:def:lsc}) 
of the limit varifold $V$ is approximately continuous. The parts where the orienting unit normal $\xi$ of $V$ is neither close to $\nu(x_0)$ nor to $-\nu(x_0)$ are then small and can be cut off. What remains splits into two parts: one where $\xi$ is close to $\nu(x_0)$, and one where $\xi$ is close to $-\nu(x_0)$.
The work to be done is to also cut away suitable parts of $V_k$ such that the remaining parts converge. The idea is that the \emph{distributional $V$-boundary} (defined in \Cref{def:lsc}\eqref{it:def:lsc:boundary}) that results from the cut off can be bounded in terms of the second fundamental form. Subsequently, the boundaries converge as Radon measures. The difficulty is that the unit normals $\nu_k$ of $V_k$ (as defined \Cref{def:vv}\eqref{it:def:vv:integral}) are not uniquely determined at points of even multiplicity. To overcome this problem, each $V_k$ will be approximated by varifolds of unit density which exist by definition of $\QV_{2,C}(\R^3)$, see \Cref{lem:lsc}. Indeed,
the unit normal
of unit density varifolds is \emph{generalised $V$-weakly differentiable} (defined in \Cref{def:lsc}\eqref{it:def:lsc:differentiability}) and generates the second fundamental form, see \Cref{prop:lsc}. The concept of second fundamental form is employed through the theory of \emph{curvature varifold with boundary} developed in \cite{Hutchinson86Indiana,Mantegazza}, see \Cref{def:lsc}\eqref{it:def:lsc:curvature}. Throughout this section, the spaces of linear maps $\R^l\to\R^m$ and bilinear maps $\R^k\times\R^l\to\R^m$ are denoted with $\R^{m\times l}$, respectively $\R^{m\times l\times k}$, where $\R^{1\times l \times k}$ is identified with $\R^{l\times k}$.

\begin{definition}[\protect{\cite[Definitions 5.1, 8.4]{MenneIndiana}, \cite[Definition 3.1]{Mantegazza}}]\label{def:lsc}
	Suppose $V\in\V_2^\uo(\R^3)$, $\|\updelta V\|$ is a Radon measure, and $\Theta^2(\mu_V,x)\ge1$ for $\mu_V$-almost all $x\in\R^3$.
	\begin{enumerate} [\upshape(1)]
		\item \label{it:def:lsc:boundary} Whenever $S\subset\R^3$ is $(\mu_V+\|\updelta V\|)$-measurable, the \emph{distributional $V$-boundary} of $S$ is given by
		\begin{equation*}
			V\partial S \vcentcolon= (\updelta V)\llcorner S - \updelta(V\llcorner S\times\G^\uo(3,2)).
		\end{equation*}
		\item \label{it:def:lsc:differentiability} A $(\mu_V+\|\updelta V\|)$-measurable $\R^p$-valued map $f$ is called \emph{generalised $V$-weakly differentiable} with \emph{$V$-weak derivative} $V\uD f$ of $f$, if $V\uD f \in L^1_\mathrm{loc}(\mu_V,\R^{p\times3})$, and for all $u\in\R^3$, $\zeta\in C^\infty_c(\R^3)$, $\gamma\in C^\infty(\R^p)$ with $\spt\uD\gamma$ compact,
		\begin{equation*}
			\updelta V((\gamma\circ f)\zeta \,u)=\int_{\G_2(\R^3)}\gamma(f(x))\uD\zeta_x(P_\natural(u))+\zeta(x)\langle\nabla\gamma(f(x)),V\uD f_x(u)\rangle\,\ud \bq_\#V(x,P).
		\end{equation*}
		\item \label{it:def:lsc:curvature} $V$ is said to be a \emph{curvature varifold with boundary} if for some $A\in L^1_\mathrm{loc}(\bq_\#V,\R^{3\times3\times3})$ called \emph{curvature}, some Radon measure $\sigma$ on $\G_2(\R^3)$, and some $\eta\in L^1_\mathrm{loc}(\partial V,\Sph^2)$,
		\begin{align*}
			&\int_{\G_2(\R^3)}\uD\varphi_{(x,P_\natural)}(P_\natural(u),\langle A(x,P),u\rangle)+\varphi(x,P) \tr(\langle A(x,P),u\rangle)\,\ud\bq_\#V(x,P) \\
			&\qquad =-\int_{\G_2(\R^3)}\varphi(x,P)\langle u,\eta(x,P)\rangle\,\ud\sigma(x,P).
		\end{align*}
		whenever $\varphi\in C^1_c(\R^3\times \R^{3\times3})$, where $\tr(\cdot)$ is the trace operator on $\R^{3\times3}$. The induced functional $C^0_c(\G_2(\R^3),\R^3)\to\R$ with $\Phi\mapsto (\sigma\llcorner \eta)(\Phi)=\int_{\G_2(\R^3)}\langle \Phi,\eta\rangle\,\ud\sigma$ is called \emph{boundary} and denoted with $\partial V$.
	\end{enumerate} 
\end{definition}

\begin{remark}\label{rem:def:lsc}
	For any volume varifold $V$ in $\R^3$, \Cref{rem:def:vv}\eqref{it:rem:def:vv:rectifiability} and \cite[Theorem 3.5(1)(b)]{Allard} imply the existence of a map $P\colon\R^3\to\G(3,2)$ such that
	\begin{equation*}
		(\bq_\#V)(\varphi)=\int_{\R^3}\varphi(x,P(x))\,\ud\mu_Vx\qquad\text{for $\varphi\in C^0(\G_2(\R^3))$}.
	\end{equation*}
	Theorem 15.6 in \cite{MenneIndiana} states that $V$ is a curvature varifold with boundary $\partial V=0$ if and only if the tangential projection $\tau\vcentcolon= P_\natural$ is $V$-weakly differentiable. In that case, $V$ has curvature $A = V\uD\tau$.
\end{remark}

\begin{prop}\label{prop:lsc}
	Suppose $(V,E)$ is a volume varifold, and $\Theta^2(\mu_V,x)=1$ for all $x\in\spt\mu_V$. Let $\nu\vcentcolon=\star\bn_E$, 
	$\tau\vcentcolon=\bp(\nu)_\natural$, $\nu_0\in\G^\uo(3,2)$, and for $y\in\R$ let $S(y)\vcentcolon=\{x\in\R^3\colon |\nu(x) - \nu_0|<y\}$, $W(y)\vcentcolon= V\llcorner S(y)\times \G^\uo(3,2)$. Then the following hold.
	\begin{enumerate} [\upshape(1)]
		\item \label{it:prop:lsc:nu} The map $\nu$ satisfies Hypothesis \eqref{it:def:vv:integral} for $(\theta_1(x),\theta_2(x))=(1,0)$, and $\nu$ is generalised $V$-weakly differentiable with $\|V\uD\nu\|\in L^2(\mu_V)$. 
		\item \label{it:prop:lsc:tau} The map $\tau$ is generalised $V$-weakly differentiable,  $C^{-1}\|V\uD\tau\|\le\|V\uD\nu\|\le C\|V\uD\tau\|$ for $C=2$, and $H_V = \tr V\uD\tau$. 
		\item \label{it:prop:lsc:curvature_varifold} For all $u\in\R^3$ and  $\SL^1$-almost all $y\ge0$, $V\partial S(y)$ is representable by integration, and
		\begin{align*}
			&\int_{\R^3}\uD\varphi_{(x,\tau(x))}(\tau(x)(u),V\uD\tau_x(u)) + \varphi(x,\tau(x)) \langle u, H_V(x)\rangle\,\ud \mu_{W(y)}x \\
			&\qquad = -\int_{\R^3}\varphi(x,\tau(x))\langle u,\eta(x)\rangle \ud\|V\partial S(y)\|x
		\end{align*}
		for all $\varphi\in C_c^1(\R^3\times\R^{3\times3})$, where $\eta$ is as in \eqref{eq:total_variation}. In particular, $W(y)$ is a curvature varifold with boundary. 
	\end{enumerate}
\end{prop}

\begin{proof}
	From \Cref{prop:vv}\eqref{it:prop:vv:integral}\eqref{it:prop:vv:normal} 
	one has $\bn_E(x)\ne0$, and $(\theta_1(x),\theta_2(x))\in\{(0,1),(1,0)\}$ for $\SH^2$-almost all $x\in\spt\mu_V$. Thus, $\nu=\star\bn_E$ is uniquely determined by Hypothesis \eqref{it:def:vv:divergence} and the divergence theorem \cite[4.5.6(5)]{Federer}. Let $\Sigma\vcentcolon=\spt\mu_V$ and $\varepsilon>0$ small enough such that $\psi=\psi(\eps)$ in \cite[Theorem 1.1]{BiZhou} satisfies $\psi<1/4$. The hypotheses \eqref{it:def:vv:mc} and $\Theta^2(\mu_V,a)=1$ for all $a\in\Sigma$ imply the existence of $r_a>0$ such that 
	\begin{equation*}
		\int_{B_{r_a}(a)}|H|^2\,\ud\mu_V<\eps^2,\qquad \mu_V(B_{r_a}(a))<(1+\eps)\pi r_a^2.
	\end{equation*} 
	Let $f_a\colon B_{r_a}\cap\R^2\to\R^3$ be as in \cite[Theorem 1.1]{BiZhou} and denote with $g_a$ its Lipschitz continuous inverse extended to all of $\R^3$ according to Kirszbraun's theorem. Define
	\begin{equation}\label{eq:prop:lsc:orientation}
		N_a(x) \vcentcolon= \Bigl(\frac{\partial_1f_a}{|\partial_1f_a|}\times\frac{\partial_2f_a}{|\partial_2f_a|}\Bigr)\circ g_a(x)\qquad \text{for $x\in f_a(B_{r_a})\cap\Sigma$}.
	\end{equation}
	After replacing $r_a$ with a smaller radius again denoted by $r_a$, one may apply \cite[Lemma 2.4]{BiZhou} combined with  Reifenberg’s topological disk theorem \cite{Reifenberg}, to first deduce that the set $\Sigma\cap B_{r_a}(a)$ is a topological disc, and then to construct a smooth spherical surface $S\subset\R^3$ containing the domain of $f_a$, and a Lipschitz continuous extension $F_a\colon S\to \R^3$ of $f_a$ such that $F_a(S)\cap B_{r_a}(a)=f_a(B_{r_a})\cap B_{r_a}(a)$.
	Hence, the same proof as in \cite[Lemma 6.6]{RuppScharrer22} shows that up to swapping the coordinates in $B_{r_a}\cap\R^2$, there holds $N_a(x) =\star\nu(x)$ for $\CH^2$-almost all $x\in B_{r(1-\psi)}(a)\cap\Sigma$. 
	Moreover, for $\SL^2$-almost all $p\in f^{-1}(B_{(1-\psi)r_a}(a))$ there holds
	\begin{equation}\label{eq:prop:lsc:projection}
		\tau(f_a(p))(u) = \sum_{i=1,2}\Bigl\langle u,\frac{\partial_if_a(p)}{|\partial_if_a(p)|}\Bigr\rangle \frac{\partial_if_a(p)}{|\partial_if_a(p)|}\qquad\text{for all $u\in\R^3$}.
	\end{equation}
	Since $\Sigma$ is compact, there exists a finite set $\Lambda\subset \{\lambda\in C^\infty_c(\R^3)\colon 0\le \lambda \le 1\}$ and corresponding $(a_\lambda,r_\lambda,f_\lambda,g_\lambda,N_\lambda)_{\lambda\in\Lambda}$ such that 
	\begin{equation}
		\Sigma\subset\bigcup_{\lambda\in\Lambda}B_{(1-\psi)r_\lambda}(a_\lambda),\qquad \sum_{\lambda\in\Lambda}\lambda(x) = 1 \quad\text{for $x\in\Sigma$},\qquad \spt \lambda \subset B_{(1-\psi)r_\lambda}(a_\lambda)\quad\text{for $\lambda\in\Lambda$}.
	\end{equation}
	Abbreviate $D_\lambda\vcentcolon=f_\lambda^{-1}(B_{(1-\psi)r_\lambda}(a_\lambda))$. By \cite[Lemma 3.2.17]{Federer}, $g_\lambda$ is differentiable relative to $\Sigma$ at $f_\lambda(p)$ for $\SL^2$-almost all $p\in D_\lambda$ with 
	\begin{equation*}
		\uD(g_\lambda|_{\Sigma})_{f_\lambda(p)}(w)=(\uD f_\lambda)_p^{-1}(w)\qquad\text{for all $w\in (\uD f_\lambda)_p(\R^2)$}.
	\end{equation*}
	It follows for $X\in C^1(\R^3,\R^3)$ that
	\begin{equation*}
		\bigl(\Div_{\nu(f_\lambda(p))}X\bigr)(f_\lambda(p)) = \langle \nabla(X\circ f_\lambda)(p),\nabla f_\lambda(p)\rangle/|\partial_1f_\lambda(p)|^2\qquad\text{for $\SL^2$-almost all $p\in D_\lambda$},
	\end{equation*}
	where the notation
	\begin{equation*}
		\langle \nabla Y,\nabla Z\rangle \vcentcolon= \langle \partial_1 Y,\partial_1Z\rangle + \langle \partial_2Y,\partial_2Z\rangle
	\end{equation*}
	was used. Thus, in view of \cite[Example 2.4]{RuppScharrer22}, the first variation formula reads
	\begin{equation}\label{eq:prop:lsc:first_variation}
		\updelta V(X)=\sum_{\lambda\in\Lambda}\int_{D_\lambda}(\lambda\circ f_\lambda) \langle \nabla(X\circ f_\lambda),\nabla f_\lambda\rangle\,\ud\SL^2  =\sum_{\lambda\in\Lambda}\int_{D_\lambda}(\lambda \langle X,H\rangle)\circ f_\lambda \,|\partial_1f_\lambda|^2\,\ud\SL^2.
	\end{equation}
	The equation remains valid for all maps $X\colon\R^3\to\R^3$ such that $X\circ f_\lambda \in W^{1,1}(D_\lambda,\R^3)$. Thus, by the chain and product rule for Sobolev functions (see \cite[Theorem 4.2.4]{EvansGariepy}),	one may take $X=(\gamma\circ (\star\nu))\zeta\,u$ where $N$ is as in \eqref{eq:prop:lsc:orientation}, $u\in\R^3$, $\zeta\in C^\infty_c(\R^3)$, and $\gamma\in C^\infty(\R^3)$ with $\spt\uD\gamma$ compact. Noting that 
	\begin{equation*}
		\star\nu(x) = N_\kappa(x) = \sum_{\lambda\in\Lambda}\lambda(x)N_\lambda(x)\qquad \text{for all $\kappa\in\Lambda$ and $x\in B_{(1-\psi)r_\kappa}(a_\kappa)\cap\Sigma$},
	\end{equation*}
	one infers from \eqref{eq:prop:lsc:projection} that
	\begin{equation*}
		\frac{\langle \nabla (X\circ f_\lambda),\nabla f_\lambda\rangle}{|\partial_1f_\lambda|^2} = (\gamma\circ N_\lambda \circ f_\lambda)\uD\zeta_{f_\lambda}(\tau(f_\lambda)(u))
		+(\zeta\circ f_\lambda)\uD\gamma_{N_\lambda(f_\lambda)}\Bigl(\sum_{i=1,2}\partial_i(N_\lambda\circ f_\lambda)\frac{\langle \partial_i f_\lambda,u\rangle}{|\partial_if_\lambda|^2}\Bigr).
	\end{equation*}
	Therefore, \eqref{eq:prop:lsc:first_variation} implies that $\star\nu$ is generalised $V$-weakly differentiable with
	\begin{equation*}
		V\uD (\star\nu)_x(u) = \sum_{\lambda\in\Lambda}\sum_{i=1,2}\lambda(x)\Bigl(\partial_i(N_\lambda\circ f_\lambda)\frac{\langle \partial_i f_\lambda,u\rangle}{|\partial_if_\lambda|^2}\Bigr)(g_\lambda(x))
	\end{equation*}
	and $\|V\uD (\star\nu)\|\in L^2(\mu_V)$ as $N_\lambda\circ f_\lambda\in W^{1,2}(D_\lambda)$. Since $\nu =\star{\star\nu}$, and the Hodge star operator is a linear isometry, this proves \eqref{it:prop:lsc:nu}, cf.\ \cite[Remark 4.10]{MenneScharrerKodai}. 
	
	Let $F\colon \textstyle\bigwedge_2 \R^3\to\R^{3\times3}$ be defined by $F(\xi)(v)=v-\langle v,\star\xi\rangle (\star\xi)$. Then, $F|_{\G^\uo(3,2)}=\bp(\cdot)_\natural$, and
	\begin{equation*}
		F(\xi+\eta)=F(\xi)-\langle\cdot,\star\eta\rangle(\star\xi) - \langle\cdot,\star\xi\rangle(\star\eta) - \langle\cdot,\star\eta\rangle(\star\eta).
	\end{equation*}
	Hence, $F$ is differentiable, $\uD F_\xi(\eta) = \langle\cdot,\star\eta\rangle(\star\xi) + \langle\cdot,\star\xi\rangle(\star\eta)$, $\uD F$ is continuous, and $\|(\uD F)|_{\G^\uo(3,2)}\|\le2$. Then, \cite[Remark 4.10]{MenneScharrerKodai} yields that $\tau=\bp(\nu)_\natural$ is generalised $V$-weakly differentiable, and $V\uD\tau = \uD F_\nu\circ V\uD \nu$. Consequently $\|V\uD\tau\|\le 2\|V\uD \nu\|$ and $\|V\uD\tau\|\in L^2(\mu_V)$. Now, \cite[Theorem 15.6]{MenneIndiana} implies $H_V = \tr V\uD\tau$.
	For each $\xi\in\G^\uo(3,2)$ let
	\begin{equation*}
		L_\xi\colon\R^{3\times3}\to \textstyle\bigwedge_2 \R^3,\qquad L_\xi(A) = \star\Bigl(A(\star\xi) - \frac{1}{2}\langle A(\star\xi),\star\xi\rangle(\star\xi)\Bigr).
	\end{equation*}
	One verifies $L_\nu\circ \uD F_\nu = \Id_{\bigwedge_2 \R^3}$, $\|L_\nu\|\le 2$, and thus $\|V\uD \nu\|\le 2\|V\uD \tau\|$ which proves \eqref{it:prop:lsc:tau}.

	Define $h\colon\R^3\to\R$ by $h(x)\vcentcolon=|\nu(x)-\nu_0|$. Then, by Part \eqref{it:prop:lsc:nu} and \cite[Lemma 8.15]{MenneIndiana}, $h$ is $V$-weakly differentiable with $\|V\uD h\|\le \|V\uD\nu\|$. Therefore, \cite[Theorem 8.30]{MenneIndiana} implies that for $\SL^1$-almost all $y\ge0$, $V\partial S(y)$ is representable by integration. For all $\delta>0$ define the Lipschitz function $s_\delta\colon \R\to\R$ by $s_\delta(t)\vcentcolon=\delta^{-1}\min\{\delta,\max\{y-t,0\}\}$ which approaches the characteristic function $\chi_{(-\infty,y)}$ from below. Testing the first variation formula \eqref{eq:prop:lsc:first_variation} with $Y=(s_\delta\circ h)X$ where $X\colon\R^3\to\R^3$ is such that $X\circ f_\lambda\in W^{1,1}(D_\lambda,\R^3)\cap L^\infty(D_\lambda,\R^3)$, one infers using the chain and product rule for Sobolev functions (see \cite[Theorem 2.1.11]{Ziemer}, \cite[Theorem 4.2.4]{EvansGariepy})
	\begin{align}
		V\partial S(y)(X) & = \lim_{\delta\to0+}\updelta V((s_\delta\circ h)X) - \lim_{\delta\to0+} \sum_{\lambda\in\Lambda}\int_{D_\lambda}(\lambda\circ f_\lambda)(s_\delta\circ h\circ f_\lambda)\langle \nabla (X\circ f_\lambda),\nabla f_\lambda\rangle\,\ud\SL^2 \nonumber \\ \label{eq:prop:lsc:distributional_boundary}
		&=-\lim_{\delta\to0+} \frac{1}{\delta}\int_{\{y-\delta\le h < y\}}V\uD h (X)\,\ud\mu_V,
	\end{align}
	where similarly,
	\begin{equation*}
		V\uD h_x(u) = \sum_{\lambda\in\Lambda}\sum_{i=1,2}\lambda(x)\Bigl(\partial_i(h\circ f_\lambda)\frac{\langle \partial_i f_\lambda,u\rangle}{|\partial_if_\lambda|^2}\Bigr)(g_\lambda(x)).
	\end{equation*}
	On the other hand, for $u\in\R^3$ and $\psi\in C^1_c(\R^3\times\R^{3\times3})$, Part \eqref{it:prop:lsc:tau} and \cite[Theorem 15.6]{MenneIndiana} imply 
	\begin{align} \nonumber
		0&=\int_{\R^3}\uD\psi_{(x,\tau(x))}(\tau(x)(u),V\uD\tau_x(u)) + \psi(x,\tau(x)) \langle u, H_V(x)\rangle\,\ud \mu_Vx\\ \label{eq:prop:lsc:curvature_varifold}
		&=\sum_{\lambda\in\Lambda}\int_{D_\lambda}\lambda(f_\lambda) \Bigl(\uD\psi_{(f_\lambda,\tau(f_\lambda))}(\tau(f_\lambda)(u),V\uD\tau_{f_\lambda}(u)) + \psi(f_\lambda,\tau(f_\lambda)) \langle u, H_V(f_\lambda)\rangle\Bigr)|\partial_1f_\lambda|^2\,\ud\SL^2.
	\end{align}
	For $\psi(x,P)=\alpha(x)\varphi(x,P)$ one computes using \eqref{eq:prop:lsc:projection}
	\begin{align*}
		&\uD\psi_{(f_\lambda,\tau(f_\lambda))}(\tau(f_\lambda)(u),V\uD\tau_{f_\lambda}(u)) \\
		&\qquad = \alpha(f_\lambda)\uD\varphi_{(f_\lambda,\tau(f_\lambda))}(\tau(f_\lambda)(u),V\uD\tau_{f_\lambda}(u)) + \varphi(f_\lambda,\tau(f_\lambda))\sum_{i=1,2}\partial_i(\alpha\circ f_\lambda)\frac{\langle \partial_i f_\lambda,u\rangle}{|\partial_i f_\lambda|^2}.
	\end{align*}
	Hence, recalling that $h\circ f_\lambda \in W^{1,2}(D_\lambda)$, and applying \eqref{eq:prop:lsc:curvature_varifold} with $\psi(x,P)=\alpha(x)\varphi(x,P)$ for $\alpha = s_\delta\circ h$ where $\varphi\in C_c^1(\R^3\times\R^{3\times3})$, one infers
	\begin{align*}
		0&=\int_{\R^3} s_\delta(h(x))\Bigl(\uD\psi_{(x,\tau(x))}(\tau(x)(u),V\uD\tau_x(u)) + \psi(x,\tau(x)) \langle u, H_V(x)\rangle\Bigr)\,\ud \mu_Vx\\ 
		&\qquad - \frac{1}{\delta}\int_{\{y-\delta\le h<y\}} \varphi(x,\tau(x)) V\uD h_x(u)\,\ud\mu_Vx.
	\end{align*}
	Letting $\delta\to0+$ and applying \eqref{eq:prop:lsc:distributional_boundary} for $X=\varphi(\cdot,\tau(\cdot))u$ proves \eqref{it:prop:lsc:curvature_varifold}.
\end{proof}

\begin{definition}\label{def:qev}
	For $0<C<\infty$, the space $\QV_{2,C}(\R^3)$ 
	is defined to consist of all volume varifolds $(V,E)$ with the following property: There exists a sequence of volume varifolds $(V_k,E_k)$ converging to $(V,E)$ according to \eqref{eq:lem:compactness:convergence_current}\eqref{eq:lem:compactness:convergence_varifold} such that $\Theta^2(\mu_{V_k},x)=1$ for all $x\in\spt\mu_{V_k}$, and $\limsup_{k \to \infty}\|V\uD \tau_k\|_{L^2(\mu_{V_k})}\le C$ for $\tau_k\vcentcolon=\bp(\star\bn_{E_k})_\natural$. Moreover, let 
	\begin{equation*}
		\QV_{2}(\R^3)\vcentcolon=\bigcup_{0<C<\infty}\QV_{2,C}(\R^3).
	\end{equation*}
\end{definition}

\begin{remark*}
	One may consider a surface as shown in \cite[Figure 2(i)]{RuppScharrer22} but with all unit normals pointing inwards to see that the class $\QV_2(\R^3)$ is strictly smaller than the class of all volume varifolds.
\end{remark*}

\begin{lemma}\label{lem:lsc}
	Suppose 
	$V\in \QV_{2,C}(\R^3)$, and $\tau$ is associated with $V$ as in \Cref{rem:def:lsc}. Then $\tau$ is generalised $V$-weakly differentiable, $\|V\uD\tau\|_{L^2(\mu_V)}\le C$, and there exists a nonnegative function $\bar A$ on $\G_2^\uo(\R^3)$ with $\|\bar A\|_{L^2(V)}\le C$ such that the following hold.
	
	For all $0<a<b<1$, Borel sets $B\subset [a,b]$ with $\SL^1(B)>0$, and $\nu_0\in\G^\uo(3,2)$, there exists $\bar y\in B$ such that the varifolds
	\begin{equation}\label{eq:lem:lsc:limit}
		W^{\pm}\vcentcolon=V\llcorner\{(x,\xi)\in\G_2^\uo(\R^3)\colon |\xi\pm\nu_0|<\bar y\}
	\end{equation}
	are curvature varifolds whose curvatures $A^\pm$ satisfy $\|A^\pm\|_{L^2(\bq_\#W^\pm)}\le C$,
	\begin{equation}\label{eq:lem:lsc:first_variation}
		\updelta W^\pm(X) = -\int_{\R^3}\langle X,H_V\rangle\,\ud\mu_{W^\pm} + \partial  W^\pm(X),
	\end{equation}
	and 
	\begin{equation}\label{eq:lem:lsc:boundary}
		\partial  W^\pm(X) \le \frac{8}{\SL^1(B)} \int_{\{(x,\xi)\in\G_2^\uo(\R^3)\colon a\le|\xi \pm \nu_0|\le b\}}|X|\,\ud V\llcorner\bar A,\qquad\text{for $X\in C^0_c(\R^3,\R^3)$}.
	\end{equation}
	Moreover, there exists a Borel function $\nu\colon \spt\mu_V\to \G^\uo(3,2)$ and integer valued $\theta_\pm\in L^1(\SH^2\llcorner\spt\mu_V)$ such that $\spt\mu_{W^\pm}=\{\theta_\pm>0\}$, 
	\begin{equation}\label{eq:lem:lsc:representation}
		W^\pm(\varphi)=\int_{\spt\mu_{W^\pm}}\varphi(x,\pm\nu(x))\theta_\pm(x)\,\ud\SH^2x\qquad \text{for all $\varphi\in C^0_c(\G_2^\uo(\R^3))$},
	\end{equation}
	and 
	\begin{equation}\label{eq:lem:lsc:decomposition}
		V=W^+ + W^- + V\llcorner\{(x,\xi)\colon d(\bp(\xi),\bp(\nu_0))\ge\bar y\},
	\end{equation}
	where the metric $d$ is as in \eqref{eq:metric}. 
\end{lemma}

\begin{proof}
	According to \Cref{def:qev}, there exists a sequence $(V_k,E_k)$ of volume varifolds such that
	\begin{equation}\label{eq:lem:lsc:convergence}
		V_k\to V\qquad \text{in $\V_2^\uo(\R^3)$ as $k\to\infty$},
	\end{equation}
	$\Theta(\mu_{V_k},x)=1$ for all $x\in\spt\mu_{V_k}$, and $\limsup_{k\in\N}\|V\uD\tau_k\|_{L^2(\mu_{V_k})} \le C$ for $\tau_k\vcentcolon=\bp(\bn_{E_k})_\natural$. From \cite[Theorem 6.1]{Mantegazza} and \cite[Theorem 15.6]{MenneIndiana} it follows that $V$ is a curvature varifold with $\partial V=0$, $\tau$ is generalised $V$-weakly differentiable, and $\|V\uD\tau\|_{L^2(\mu_V)}\le\liminf_{k\to\infty}\|V\uD\tau_k\|_{L^2(\mu_{V_k})}$. 
	Define the oriented varifolds $\Psi_k\vcentcolon=V_k\llcorner\|V\uD\tau_k\|$. Then, by the compactness of Radon measures \cite[2.6(a)]{Allard}, and \cite[Theorem 2.2]{ButtazzoFreddi}, there exists $0\le \bar A\in L^2(V)$ with $\|\bar A\|_{L^2(V)} \le \liminf_{k\to\infty}\|V\uD\tau_k\|_{L^2(\mu_{V_k})}$ and $\Psi_k\to V\llcorner \bar A$ in $\V^\uo_2(\R^3)$ as $k\to\infty$.
	
	Let $G$ be the set of $y\in[0,1]$ such that $\mu_V(\{x\in\R^3\colon d(P(x),P_0) =y\})=0$, where $d$ is the metric defined in \eqref{eq:metric}, $P$ is as in \Cref{rem:def:lsc}, and $P_0=\bp(\nu_0)$. Obviously, up to a countable set, $G$ equals $[0,1]$, and $\SL^1(B)=\SL^1(B\cap G)$. Hence, there exists a compact set $K\subset B\cap G$ such that $\SL^1(K)\ge \SL^1(B)/2$.
	For all $\eps>0$, choose $\chi_\eps$ to be a continuous function on $\R$ satisfying $0\le\chi_\eps\le1$, and
	\begin{equation*}
		\chi_\eps(y)=
			\begin{cases*}
				0&\text{for $y<a-\eps$},\\
				1&\text{for $a \le y \le b$},\\
				0&\text{for $y >b+\eps$}.
			\end{cases*}
	\end{equation*}
	For each $y\in\R$ and all $k\in\N$ define $\nu_k\vcentcolon=\star\bn_{E_k}$, and
	\begin{equation*}
		S^\pm_k(y)\vcentcolon=\{x\in\R^3\colon |\nu_k(x)\pm\nu_0|<y\}.
	\end{equation*}
	In view of \Cref{prop:lsc}\eqref{it:prop:lsc:nu} and \cite[Remark 8.16]{MenneIndiana}, one may apply \cite[Remark 8.5, Theorem 8.30]{MenneIndiana} for $g(x,y)=|X(x)|\chi_\eps(y)$ to infer the existence of $y_k\in K$ such that, using \Cref{def:vv}\eqref{it:def:vv:integral},
	\begin{align}\nonumber
		V_k\partial S^\pm_k(y_k)(X) &\le \frac{2}{\SL^1(K)}\int_a^b\int_{\R^3}|X(x)|\chi_\eps(y)\,\ud\|V_k\partial S_k^\pm(y)\|(x)\,\ud y\\ \nonumber
		&\le \frac{4}{\SL^1(B)}	\int_{\G_2^\uo(\R^3)}|X(x)|\chi_\eps(|\nu_k(x)\pm\nu_0|)\|V\uD(\nu_k)_x\|\,\ud \mu_{V_k}x\\ \nonumber
		&\le \frac{8}{\SL^1(B)}	\int_{\G_2^\uo(\R^3)}|X(x)|\chi_\eps(|\xi\pm\nu_0|)\|V\uD(\tau_k)_x\|\,\ud V_k(x,\xi)\\ \label{eq:lem:lsc:boundary_bound}
		&=\frac{8}{\SL^1(B)}\int_{\G_2^\uo(\R^3)}|X(x)|\chi_\eps(|\xi\pm\nu_0|)\,\ud \Psi_k(x,\xi).
	\end{align}
	After passing to a subsequence, $y_k$ converges to some $\bar y\in K\subset B$. Define the sequence of oriented varifolds
	\begin{equation*}
		W_k^\pm\vcentcolon=V_k\llcorner S_k^\pm(y_k)\times\G^\uo(3,2).
	\end{equation*}
	Using \eqref{eq:lem:lsc:convergence} and the fact that $\mu_V(\{x\in\R^3\colon d(P(x),P_0)=\bar y\})=0$, one may employ continuous cut-off functions on $\G^\uo(3,2)$ and a standard limit theorem \cite[2.1.3(5)]{Federer} to see that
	\begin{equation}\label{eq:lem:lsc:restricted_convergence}
		W_k^\pm \to W^\pm \qquad\text{in $V_2^\uo(\R^3)$ as $k\to\infty$},
	\end{equation}
	where $W^\pm$ are defined as in \eqref{eq:lem:lsc:limit}. By \Cref{prop:lsc}\eqref{it:prop:lsc:curvature_varifold} and \eqref{eq:lem:lsc:boundary_bound}, $W_k^\pm$ are curvature varifolds whose boundaries satisfy
	\begin{equation*}
		|\partial W_k^\pm| \le \frac{8}{\SL^1(B)}\int_{\R^3}\|V\uD\tau_k\|\,\ud\mu_{V_k} \le \frac{8}{\SL^1(B)}\|V\uD\tau_k\|_{L^2(\mu_{V_k})}\sqrt{\CA(V_k)}.
	\end{equation*}
	The right hand side is bounded since $\lim_{k\to\infty}\CA(V_k)=\CA(V)$ and $\limsup_{k\in\N}\|V\uD\tau_k\|_{L^2(\mu_{V_k})} \le C$. Hence, by \cite[Theorem 6.1]{Mantegazza}, $\bq_\#W_k^\pm$ converge in $\V_2(\R^3)$ to some curvature varifolds $\bar W^\pm$ such that the boundaries $\partial W_k^\pm$ weakly converge to $\partial \bar W^\pm$, and 
	by lower semi-continuity, the curvatures $A^\pm$ satisfy $\|A^\pm\|_{L^2(\bq_\#W^\pm)}\le C$.
	From \eqref{eq:lem:lsc:restricted_convergence} it follows $\bar W^\pm = \bq_\#W^\pm$. 
	Equation \eqref{eq:lem:lsc:first_variation} then follows from \cite[Proposition 3.10]{Mantegazza} and the locality of the mean curvature, see instance \cite[Theorem~1]{MenneJGA} and \cite[Theorem 4.1]{SchatzleJDG}. The convergence of $\Psi_k$ and \eqref{eq:lem:lsc:boundary_bound} imply
	\begin{align*}
		\partial W^\pm(X) & = \lim_{k\to\infty}\partial W^\pm_k(X) \le \liminf_{k\to\infty} \frac{8}{\SL^1(B)}\int_{\G_2^\uo(\R^3)}|X(x)|\chi_\eps(|\xi\pm\nu_0|)\,\ud \Psi_k(x,\xi) \\
		& = \frac{8}{\SL^1(B)}\int_{\G_2^\uo(\R^3)}|X(x)|\chi_\eps(|\xi\pm\nu_0|)\,\ud (V\llcorner \bar A)(x,\xi) \\
		& \le \frac{8}{\SL^1(B)}\int_{\{(x,\xi)\in\G_2^\uo(\R^3)\colon a-\eps\le|\xi\pm\nu_0|\le b+\eps\}}|X|\,\ud V\llcorner \bar A.
	\end{align*}
	Letting $\eps\to0$ implies \eqref{eq:lem:lsc:boundary}. By  \cite[Proposition 3.10]{Mantegazza} the mean curvature is given in terms of the curvature. Thus, one may apply the compactness theorem for integral varifolds \cite[6.4]{Allard} to deduce that $\Theta^2(\mu_{W^\pm},\cdot)$ are integer valued functions. Hence, \eqref{eq:lem:lsc:representation} follows from \cite[Theorem 3.5(1)]{Allard} by noting that the the maps $\bp|_{K^\pm}$ with $\bp$ as in \eqref{eq:projection} and $K^\pm\vcentcolon=\{\xi\in\G^\uo(3,2)\colon|\xi\pm\nu_0|\le\bar y\}$ have a continuous inverse. Equation \eqref{eq:lem:lsc:decomposition} is obvious by definition of $W^\pm$.
\end{proof}

\begin{theorem}\label{thm:lsc}
	Suppose $0<C<\infty$, $V_k$ is a sequence in $\QV_{2,C}(\R^3)$, $\limsup_{k\to\infty}\CA(V_k)<\infty$, and $\spt\mu_{V_k}$ is connected with $0\in\spt\mu_{V_k}$ for all $k\in\N$.
	Then, after passing to a subsequence, $V_k$ converges to some $V\in\QV_{2,C}(\R^3)$ in the sense of \eqref{eq:lem:compactness:convergence_current}\eqref{eq:lem:compactness:convergence_varifold}, and
	\begin{equation*}
		\CH_{c_0}(V) \le \liminf_{k\to\infty}\CH_{c_0}(V_k)\qquad \text{for all $c_0\in\R$}.
	\end{equation*}
\end{theorem}

\begin{proof}
	Given any $\bar C<\infty$, \Cref{lem:compactness} and \cite[Theorem 6.1]{Mantegazza} imply that the space of volume varifolds $W$ that are curvature varifolds with curvature $\|A_W\|_{L^2(\mu_W)}\le C$ and $\CA(W)\le \bar C$ is compact with respect to the convergence \eqref{eq:lem:compactness:convergence_current}\eqref{eq:lem:compactness:convergence_varifold}. Since the convergence of Radon measures \eqref{eq:weak_convergence} is induced by a metric (cf.\ \cite[Example 2.23]{MenneIndiana}), one infers that after passing to a subsequence, $V_k$ converges to some $V\in \QV_{2,C}(\R^3)$ in the sense of \eqref{eq:lem:compactness:convergence_current}\eqref{eq:lem:compactness:convergence_varifold}. Let $\bar A_k$ be associated with $V_k$ as in \Cref{lem:lsc}, and $\Psi_k\vcentcolon=V_k\llcorner \bar A_k$. By \cite[Theorem 2.2]{ButtazzoFreddi}, $\Psi_k$ converges in $\V_2^\uo(\R^3)$ to $V_\llcorner \bar A$ for some nonnegative $\bar A\in L^2(V)$. Let $c_0\in\R$. Define the Radon measure $\psi$ on $\R^3$ by
	\begin{equation*}
		\psi(B)\vcentcolon=\int_{B\times\G^\uo(3,2)}|H_V(x)-c_0(\star\xi)|^2\,\ud V(x,\xi)\qquad\text{for all Borel sets $B\subset \R^3$}.
	\end{equation*} 	
	Then $\psi$ is absolutely continuous with respect to $\mu_V$. In particular, the tangent plane map $P$ of $V$ as defined in \Cref{rem:def:lsc} is $\psi$-measurable. Thus, by \cite[Theorem 2.9.13]{Federer}, the set $G$ of points $a\in\R^3$ at which $P$ has a $\psi$-approximate limit $P_a$ satisfies $\psi(\R^3\setminus G)=0$. Let $\eps>0$. It will be shown that
	\begin{equation*}
		\CH_{c_0}(V)\le\liminf_{k\to\infty}\CH_{c_0}(V_k) + \eps. 
	\end{equation*}
	Choose $\bar\eps>0$ such that $\bar\eps\psi(\R^3)<\eps$. Let $a\in G$ and $P_a$ be the $\psi$-approximate limit of $P$ at $a$. Then, there exists $r_a>0$ such that for the metric $d$ as defined in \eqref{eq:metric},
	\begin{equation*}
		\psi\bigl(\{x\in \bar B_{r}(a)\colon d(P(x),P_a)\ge1/3\}\bigr)\le \bar \eps\psi(\bar B_{r}(a))\qquad \text{for all $0<r<r_a$}.
	\end{equation*}
	In a first step it will be shown that 
	\begin{equation}\label{eq:thm:lsc:local}
		\psi(B_{r}(a)) \le \liminf_{k\to\infty}\int_{B_{r}(a)\times\G^\uo(3,2)}|H_k(x)-c_0(\star\xi)|^2\,\ud V_k(x,\xi) + \bar \eps\psi(B_{r}(a))
	\end{equation}
	for all $0<r<r_a$ with $\mu_V(\partial B_{r}(a))=0$. Let $B\subset [\frac{1}3,\frac{2}3]$ be a compact set such that $\SL^1(B)=\frac{1}{4}$ and  $\mu_V(\{x\in\R^3\colon d(P(x),P_a)=y\})=0$ for all $y\in B$. Let $\nu_a$ be any member of $\G^\uo(3,2)$ with $\bp(\nu_a)=P_a$ and choose $\bar y_k \in B$ related to $V_k$ and $\nu_a$ as in \Cref{lem:lsc}. After passing to a subsequence, $\bar y_k$ converges to some $\bar y\in B$. Define the oriented varifolds
	\begin{equation*}
		W_k^\pm\vcentcolon=V_k\llcorner\{(x,\xi)\in\G^\uo(3,2)\colon|\xi\pm\nu_a|<\bar y_k\},\quad W^\pm\vcentcolon=V\llcorner\{(x,\xi)\in\G^\uo(3,2)\colon|\xi\pm\nu_a|<\bar y\}.
	\end{equation*}
	The convergence of $V_k$ to $V$ in $\V_2^\uo(\R^3)$ and the definition of the set $B$ imply that
	\begin{equation}\label{eq:thm:lsc:oriented_convergence}
		W_k^\pm\to W^\pm\qquad \text{in $\V_2^\uo(\R^3)$ as $k\to\infty$}. 
	\end{equation}
	Moreover, $W_k^\pm$ are curvature varifolds with curvature $A_k^\pm$ such that
	\begin{equation*}
		\sup_{k\in\N}\|A_k^\pm\|_{L^2(\bq_\#W_k^\pm)}\le C,\qquad \limsup_{k \to \infty}|\partial W^\pm_k|\le 32\lim_{k\to \infty}\sqrt{\CA(V_k)}C.
	\end{equation*}
	Hence, \cite[Theorem 6.1]{Mantegazza} implies that $\bq_\#W_k^\pm$ converge in $\V_2(\R^3)$ to a curvature varifold $\bar W^\pm$ such that $\partial W_k^\pm$ weakly converge to $\partial \bar W^\pm$. From \eqref{eq:thm:lsc:oriented_convergence} it follows $\bar W^\pm=\bq_\#W^\pm$. Analogously to the proof of \Cref{lem:lsc}, it follows that $W^\pm$ satisfy the Equations \eqref{eq:lem:lsc:first_variation}\eqref{eq:lem:lsc:representation}\eqref{eq:lem:lsc:decomposition} for some Borel function $\nu\colon\spt\mu_V\to\G^\uo(3,2)$ and some integer valued $\theta_\pm\in L^1(\SH^2\llcorner\spt\mu_V)$ with $\spt\mu_{W^\pm}=\{\theta_\pm>0\}$. Given any $X\in C^\infty_c(B_{r}(a))$, it follows
	\begin{align*}
		&\int_{B_r(a)}\langle X(x),H_V(x)\mp c_0\nu(x)\rangle \theta_\pm(x)\,\ud\SH^2x \\ 
		&\qquad = \partial W^\pm(X)- \updelta W^\pm(X) - c_0\int_{B_r(a)\times\G^\uo(3,2)}\langle X(x),\star\xi\rangle\,\ud W^\pm(x,\xi) \\
		&\qquad = \lim_{k\to \infty}\Bigl(\partial W_k^\pm(X)- \updelta W_k^\pm(X) -c_0\int_{B_r(a)\times\G^\uo(3,2)}\langle X(x),\star\xi\rangle\,\ud W_k^\pm(x,\xi)\Bigr) \\
		&\qquad = \lim_{k\to \infty} \int_{B_r(a)}\langle X(x),H_k(x)- c_0(\star\xi)\rangle \,\ud W^\pm_k(x,\xi) \\
		&\qquad \le \|X\|_{L^2(\mu_{W^\pm})} \liminf_{k\to \infty}\Bigl(\int_{B_r(a)}|H_k(x)- c_0(\star\xi)|^2\,\ud W^\pm_k(x,\xi)\Bigr)^\frac{1}{2}.
	\end{align*}
	Hence, from \cite[Theorem 3.4]{MS23a} it first follows that 
	\begin{align*}
		\int_{B_r(a)}|H_V(x)- c_0(\star\xi)|^2\,\ud W^\pm(x,\xi)&=\int_{B_r(a)}|H_V(x)\mp c_0\nu(x)|^2 \theta_\pm(x)\,\ud\SH^2x \\
		&\le \liminf_{k\to \infty} \int_{B_r(a)}|H_k(x) - c_0(\star\xi)|^2\,\ud W^\pm_k(x,\xi),
	\end{align*}
	and then
	\begin{align*}
		&\int_{B_r(a)}|H_V(x) - c_0(\star\xi)|^2\,\ud V (x,\xi) \\
		&\qquad = \int_{B_r(a)}|H_V(x) - c_0(\star\xi)|^2\,\ud W^+(x,\xi) + \int_{B_r(a)}|H_V(x) - c_0(\star\xi)|^2\,\ud W^-(x,\xi) \\
		&\qquad \qquad +\psi\bigl(\{x\in B_{r}(a)\colon d(P(x),P_a)\ge\bar y\}\bigr) \\
		&\qquad \le \liminf_{k\to\infty}\int_{B_r(a)\times\G^\uo(3,2)}|H_k(x)-c_0(\star\xi)|^2\,\ud V_k(x,\xi) + \bar\eps\psi(\bar B_r(a))
	\end{align*}
	which is \eqref{eq:thm:lsc:local}. 
	
	Now, the covering theorem \cite[2.8.18]{Federer} implies the existence of a countable set  
	\begin{equation*}
		J\subset\{(a,r)\in G\times(0,\infty)\colon0<r<r_a\}
	\end{equation*}
	such that the members of $\{\bar B_{r}(a)\colon(a,r)\in J\}$ are pairwise disjoint, $\psi(\R^3\setminus\bigcup_{(a,r)\in J}\bar B_r(a))=0$, and $\mu_V(\partial B_r(a))=0$ for all $(a,r)\in J$. It follows that
	\begin{align*}
		\psi(\R^3)&=\sum_{(a,r)\in J}\psi(B_r(a))\le\sum_{(a,r)\in J}\Bigl(\liminf_{k\to\infty}\int_{B_r(a)\times\G^\uo(3,2)}|H_k(x)-c_0(\star\xi)|^2\,\ud V_k(x,\xi) + \bar\eps\psi(B_r(a))\Bigr)\\
		&\le \liminf_{k\to \infty}\Bigl(\sum_{(a,r)\in J}\int_{B_r(a)\times\G^\uo(3,2)}|H_k(x)-c_0(\star\xi)|^2\,\ud V_k(x,\xi)\Bigr) +\bar \eps\psi(\R^3)\\
		&\le \liminf_{k\to \infty}\int_{\G^\uo_2(\R^3)}|H_k(x)-c_0(\star\xi)|^2\,\ud V_k(x,\xi) +\eps.
	\end{align*}
	Letting $\eps\to0$ concludes the proof.
\end{proof}

\section{Lipschitz quasi embeddings}\label{sec:Lipschitz_embeddings}

In this section, \Cref{T} will be proved (see \Cref{cor:lsc}) by applying the lower semi-continuity of the previous section to the theory of Lipschitz immersions which are introduced at the beginning of this section. The main references are \cite{KuwertLi12CAG,ChenLi14AJM,Riviere14Crelle,RuppScharrer22}.

Let $\Sigma$ be a closed connected oriented surface with Riemannian metric $h$. A map $f\in W^{2,2}(\Sigma,\R^3)$ is called \emph{weak branched immersion}, if there exists $0<C<\infty$ such that
\begin{equation}\label{eq:lqe:immersion}
	C^{-1}|\ud f|_h\le|\ud f\wedge\ud f|_h\le C|\ud f|_h,
\end{equation}
where in local coordinates
\begin{equation*}
	\ud f\wedge\ud f\vcentcolon=(\ud x^1\wedge\ud x^2)\partial_{x^1}f\wedge\partial_{x^2}f,
\end{equation*}
there exist finitely many \emph{branch points} $b_1,\ldots,b_N\in\Sigma$ such that
\begin{equation}\label{eq:lqe:conformal_factor}
	\log |\ud f|_h\in L^\infty_\mathrm{loc}(\Sigma\setminus\{b_1,\ldots,b_N\}), 
\end{equation}
and the \emph{Gauss map} satisfies
\begin{equation}\label{eq:lqe:Gauss_map}
	n\in W^{1,2}(\Sigma),
\end{equation}
where in local positive coordinates
\begin{equation*}
	n\vcentcolon=\frac{\partial_{x^1}f\times\partial_{x^2}f}{|\partial_{x^1}f\times\partial_{x^2}f|}.
\end{equation*}
If in addition 
\begin{equation*}
	|\partial_{x^1}f|=|\partial_{x^2}f|,\qquad \langle \partial_{x^1}f,\partial_{x^2}f\rangle = 0
\end{equation*}
for all conformal charts $x$ of $(\Sigma,h)$, then $f$ is called \emph{conformal}. The space of Lipschitz-continuous weak branched immersions is denoted with $\CF_\Sigma$. 
The space of \emph{Lipschitz immersions} $\CE_\Sigma$ is the subspace of $\CF_\Sigma$ consisting of all \emph{unbranched} maps. To be more precise, $\CE_\Sigma$ consists of all $f\in\CF_\Sigma$ with $\log|\ud f|_h\in L^\infty(\Sigma)$. Indeed, if $f$ is a conformal Lipschitz immersion, then the conditions \eqref{eq:lqe:immersion}\eqref{eq:lqe:Gauss_map} as well as the Lipschitz-continuity are guaranteed by the single requirement that the \emph{conformal factor} satisfies $\log|\ud f|_h\in L^\infty(\Sigma)$.

Let $f\in\CF_\Sigma$. The pull-back metric $g_f\vcentcolon=f^*\langle\cdot,\cdot\rangle$ of the Euclidean metric along $f$ induces a Radon measure $\mu_f$ on $\Sigma$. In view of \cite[Example 2.4, Section 6.3]{RuppScharrer22}, the oriented varifold $V_f\vcentcolon=(f,\star n)_\#\mu_f$ satisfies Hypotheses \eqref{it:def:vv:mass}-\eqref{it:def:vv:integral}. The space $\CQ_\Sigma$ of \emph{Lipschitz quasi embeddings} consists of all $f\in\CE_\Sigma$ such that $V_f$ is a volume varifold. For all $f\in \CQ_\Sigma$, $c_0\in\R$, and $x\in\R^3$ let
\begin{equation*}
	\CA(f)\vcentcolon=\CA(V_f),\quad \CV(f)\vcentcolon=\CV(V_f),\quad \CH_{c_0}(f)\vcentcolon=\CH_{c_0}(V_f), \quad \CV_c(f,x)\vcentcolon=\CV_c(V_f,x).
\end{equation*}

\begin{cor}\label{cor:lsc}
	For each nonnegative integer $g$ let $\Sigma_g$ be a closed connected oriented surface of genus $g$. Let $G$ be a nonnegative integer, $c_0\in\R$, $a_0,v_0>0$, and let
	\begin{equation*}
		\eta_G(c_0,a_0,v_0)\vcentcolon=\min_{0\le g\le G}\inf\{\CH_{c_0}(f)\colon f\in \CQ_{\Sigma_g},\,\CA(f)=a_0,\,\CV(f)=v_0\}.
	\end{equation*}
	If there exists a minimising sequence $f_k$ of the infimum $\eta_G(c_0,a_0,v_0)$ 
	such that 
	\begin{align} \label{eq:cor:lsc:negative}
		\liminf_{k\to\infty}\bigl(\CH_{c_0}(f_k) +2c_0\inf_{x\in \im f_k}\CV_c(f_k,x)\bigr) < 8\pi &\quad  \text{if $c_0<0$},\\ \label{eq:cor:lsc:positive}
		\liminf_{k\to\infty}\bigl(\CH_{c_0}(f_k) +2c_0\sup_{x\in \im f_k}\CV_c(f_k,x)\bigr) < 8\pi & \quad \text{if $c_0\geq0$},
	\end{align} 
	then the infimum $\eta_G(c_0,a_0,v_0)$ is attained by a smooth embedding $f\colon\Sigma_g\to\R^3$ for some $0\le g\le G$.
\end{cor}

\begin{remark}\label{rem:cor:lsc}\leavevmode
	Obviously, a minimising sequence can only exist if $a_0, v_0$ satisfy the isoperimetric inequality $v_0 \le C a_0^{3/2}$ of \Cref{prop:vv}\eqref{it:prop:vv:iso-ineq}.
	
	The following proof shows that by 
	\cite[Lemmas 3.6,\,3.7]{RuppScharrer22}, and \Cref{thm:diam}, Condition \eqref{eq:cor:lsc:negative} is guaranteed by the a priori requirement that $\eta\vcentcolon=\eta_G(c_0,a_0,v_0)$ satisfies
	\begin{equation*}
		\eta - \frac{2|c_0|v_0}{C^2_{\ref{thm:diam}}a_0\eta} < 8\pi.
	\end{equation*} 
	This can be rephrased as 
	\begin{equation*}
		\eta<4\pi\Bigl(1+\sqrt{1+L(c_0,a_0,v_0)}\Bigr)\qquad \text{for $L(c_0,a_0,v_0)\vcentcolon=\frac{|c_0|v_0}{8\pi^2C^2_{\ref{thm:diam}}a_0}$}.
	\end{equation*}
	
	In view of \cite[Remark 6.11(ii)]{RuppScharrer22}, Condition \eqref{eq:cor:lsc:positive} is guaranteed by the a priori requirement
	\begin{equation*}
		\eta < 8\pi -6c_0(4\pi^2v_0)^{1/3}.
	\end{equation*} 
\end{remark}

\begin{proof}
	Let $f_k\colon\Sigma_{g_k}\to\R^3$ be a minimising sequence of $\eta\vcentcolon=\eta_G(c_0,a_0,v_0)$ satisfying \eqref{eq:cor:lsc:negative}\eqref{eq:cor:lsc:positive}. After adding a translating constant to each $f_k$, one has $0\in\im f_k = \spt\mu_k$. Moreover, after passing to a subsequence, there exists $0\le g \le G$ such that $g_k=g$ for all $k\in\N$, and
	\begin{equation}\label{eq:cor:lsc:ad_hoc_bound}
		\begin{cases}
			\CH_{c_0}(f_k) +2c_0\inf_{x\in f_k[\Sigma_g]}\CV_c(f_k,x) < 8\pi & \text{if $c_0<0$},\\
			\CH_{c_0}(f_k) +2c_0\sup_{x\in f_k[\Sigma_g]}\CV_c(f_k,x) < 8\pi & \text{if $c_0\geq0$}.
		\end{cases}
	\end{equation}
	Therefore, by \cite[Corollary 4.4]{RuppScharrer22}, the induced oriented varifolds $V_k\vcentcolon=V_{f_k}$ satisfy $\Theta^2(\mu_{V_k},x)=1$ for all $x\in\spt\mu_{V_k}$. Hence, \Cref{prop:lsc} implies that $V_k$ is a sequence in $\QV_{2}(\R^3)$. By \cite[Equation (2.8)]{MondinoScharrer} and \Cref{prop:vv}\eqref{it:prop:W_by_H}, there exists a constant $C\vcentcolon=C(G,\eta,c_0,a_0)$ depending only on $G,\eta,c_0,a_0$ such that $V_k$ is in fact a sequence in $\QV_{2,C}(\R^3)$. Hence, \Cref{thm:lsc} implies that $V_k$ converges in $\V_2^\uo(\R^3)$ to some 
	$V\in \QV_{2,C}(\R^3)$ with $\CH_{c_0}(V)\le\eta$.
	In particular, the weight measures converge:
	\begin{equation}\label{eq:cor:lsc:varifold_convergence}
		\lim_{k\to \infty}\int_{\R^3}\varphi\,\ud \mu_{V_k}=\int_{\R^3}\varphi\,\ud \mu_{V}\qquad\text{for all $\varphi\in C_c^0(\R^3)$}.
	\end{equation}
	Moreover, the lower semi-continuity \Cref{thm:lsc}, \eqref{eq:cor:lsc:ad_hoc_bound}, and \cite[Lemma 3.6, Corollary 4.4]{RuppScharrer22} imply
	\begin{equation}\label{eq:cor:lsc:density1}
		\Theta^2(\mu_V,x)=1\qquad\text{for all $x\in\spt\mu_V$}.
	\end{equation}

	On the other hand, \cite[Theorem 1.4]{Riviere14Crelle} implies that each $f_k$ induces a conformal structure $h_k$ on $\Sigma_g$ such that, after reparametrisation, all $f_k\colon(\Sigma_g,h_k)\to\R^3$ are conformal. Therefore, by \cite[Definition 3, Theorem 1]{ChenLi14AJM}, there exists a metric space $\Sigma$ given as the union of finitely many closed oriented connected surfaces $\overline\Sigma^1,\ldots,\overline\Sigma^N$ and a map $f\colon\Sigma\to\R^3$ with $f|_{\overline\Sigma^i}\in\CF_{\overline\Sigma^{i}}$ for $i=1,\ldots,N$
	such that $f(\Sigma)$ is connected and the induced measures $\mu_{f_k}$ converge as Radon measure to $\mu_{f}=\sum_{i=1}^N\mu_{f|_{\overline\Sigma^i}}$. From \eqref{eq:cor:lsc:varifold_convergence} and \eqref{eq:cor:lsc:density1} it first follows that $f_\#\mu_f=\mu_V$ and $\Theta^2(f_\#\mu_{f},x)=1$ for all $x\in f(\Sigma)$, and then $\Sigma = \overline \Sigma^1$. Moreover, since $f$ is injective, it has no branch points (see \cite[Theorem 3.1]{KuwertLi12CAG}), thus is a member of $\CE_{\overline\Sigma^{1}}$. By \cite[Theorem 1]{ChenLi14AJM}, the genus of $\overline \Sigma^1$ is bounded above by $g$. Moreover, by \cite[Lemma 6.6]{RuppScharrer22}, $f\in \CQ_\Sigma$. Hence $f$, is in the same class as the minimising sequence. The set $E$ of \Cref{def:vv} is determined by the Jordan--Brouwer separation theorem which in turn determines the unit normal. 
	Thus, $V_f=V$ and $\eta\le\CH_{c_0}(f)=\CH_{c_0}(V)\le\eta$.
	
	Now, one can combine \cite[Lemmas 3.6, 3.7, 6.6, Corollary 4.4]{RuppScharrer22} to see that local variations of~$f$ stay in the same class as the minimising sequence. Hence, $f$ is smooth as a consequence of \cite[Theorem 4.3]{MondinoScharrer}.
\end{proof}

\bibliography{mybib}
\bibliographystyle{alpha}
\end{document}